\documentclass[11pt,a4paper,reqno]{amsart}
\usepackage{geometry}
\usepackage[T1]{fontenc}

\usepackage{cite}
\usepackage{amssymb, amsmath, amsthm, enumerate,  mathtools}
\usepackage{color}
\usepackage[dvipsnames]{xcolor}
\usepackage{graphicx}
\usepackage{tikz}

\usepackage{subcaption} 
\usepackage{caption}

\usepackage[thinlines]{easytable}

\usepackage{mathabx} 
\usepackage{esint} 
\usepackage{bm}
\usepackage{tikz}
\usepackage{tikz-3dplot}
\usepackage[normalem]{ulem}

\usetikzlibrary{calc,patterns,angles,quotes}

\usepackage{hyperref}
\usepackage{cite}

\usepackage{etoolbox}
\newtheorem{theorem}{Theorem}[section]
\newtheorem{lemma}[theorem]{Lemma}
\newtheorem{proposition}[theorem]{Proposition}
\newtheorem{corollary}[theorem]{Corollary}
\newtheorem{definition}[theorem]{Definition}
\newtheorem{claim}[theorem]{Claim}
\newtheorem{remark}[theorem]{Remark}

\newtheorem{conjecture}[theorem]{Conjecture}

\preto{\section}{}
\preto{\subsection}{}

\makeatletter
\@addtoreset{theorem}{section}
\@addtoreset{theorem}{subsection}
\makeatother

\usepackage{url}

\begin{document}

\title[Some sharp inequalities of Mizohata--Takeuchi-type]{Some sharp inequalities of Mizohata--Takeuchi-type}
\author{Anthony Carbery, Marina Iliopoulou and Hong Wang}
\address{School of Mathematics and Maxwell Institute for Mathematical Sciences\\
University of Edinburgh\\
Edinburgh EH9 3FD\\
Scotland}
\email{A.Carbery@ed.ac.uk}

\address{Department of Mathematics\\
National and Kapodistrian University of Athens\\
Zografou\\
Athens  157 84\\
Greece
}
\email{miliopoulou@math.uoa.gr}

\address{ Courant institute of Mathematical Sciences\\
New York University\\
New York, NY 10012\\
USA}
\email{hw3639@nyu.edu}

\begin{abstract}
    Let $\Sigma$ be a strictly convex, compact patch of a $C^2$ hypersurface in $\mathbb{R}^n$, with non-vanishing Gaussian curvature and surface measure $d\sigma$ induced by the Lebesgue measure in $\mathbb{R}^n$. The Mizohata--Takeuchi conjecture states that
    \begin{equation*}
        \int |\widehat{gd\sigma}|^2w \leq C \|Xw\|_\infty \int |g|^2
    \end{equation*}
    for all $g\in L^2(\Sigma)$ and all weights $w:\mathbb{R}^n\rightarrow [0,+\infty)$, where $X$ denotes the $X$-ray transform. As partial progress towards the conjecture, we show, as a straightforward consequence of recently-established  decoupling inequalities, that 
    for every $\epsilon>0$, there exists a positive constant $C_\epsilon$, which depends only on $\Sigma$ and $\epsilon$, such that for all $R \geq 1$ and all weights $w:\mathbb{R}^n\rightarrow [0,+\infty)$ we have
\begin{equation*}
        \int_{B_R} |\widehat{gd\sigma}|^2w \leq C_\epsilon R^\epsilon \sup_T \left(\int _T w^{\frac{n+1}{2}}\right)^{\frac{2}{n+1}}\int |g|^2,
    \end{equation*}
 where $T$ ranges over the family of all tubes in $\mathbb{R}^n$ of dimensions $R^{1/2} \times \dots \times R^{1/2} \times R$. From this we deduce the Mizohata--Takeuchi conjecture with an $R^{\frac{n-1}{n+1}}$-loss; i.e., that
\begin{equation*}
        \int_{B_R} |\widehat{gd\sigma}|^2w \leq C_\epsilon R^{\frac{n-1}{n+1}+ \epsilon}\|Xw\|_\infty\int |g|^2
    \end{equation*}
    for any ball $B_R$ of radius $R$ and any $\epsilon>0$. The power $(n-1)/(n+1)$ here cannot be replaced by anything smaller unless properties of $\widehat{gd\sigma}$ beyond `decoupling axioms' are exploited. We also provide estimates which improve this inequality under various conditions on the weight, and discuss some new cases where the conjecture holds.  
\end{abstract}

\maketitle

\section{Introduction}

Let $n\geq 2$, and henceforth fix $\Sigma$ to be a strictly convex, compact patch of a $C^2$ hypersurface in $\mathbb{R}^n$ with non-vanishing Gaussian curvature; a prototypical example is the sphere $\mathbb{S}^{n-1}$. Let $d\sigma$ be the surface measure on $\Sigma$, induced by the Lebesgue measure in $\mathbb{R}^n$. The \textit{Fourier extension operator} associated to $\Sigma$ is defined by
\begin{equation*}
    g\mapsto \widehat{gd\sigma}
\end{equation*}
where
\begin{equation*}
    \widehat{gd\sigma}(x):=\int e^{2\pi i \langle x,\xi\rangle} g(\xi)d\sigma(\xi) \text{ for  }x\in\mathbb{R}^n.
\end{equation*}
The Fourier restriction or extension conjecture \cite{St78}, which lies at the heart of harmonic analysis, aims to understand the extension operator by determining its $L^p\rightarrow L^q$ mapping properties. However, while Fourier extension estimates provide information on the size of the level sets of $|\widehat{gd\sigma}|$, they do not reveal much about their shape. The Mizohata--Takeuchi conjecture  aims to shed light in this direction, specifically regarding the clustering of level sets along lines. The conjecture arose in the study of dispersive PDE; see \cite{Mi85} for some background. In that setting, hypersurfaces such as the paraboloid and the cone are particularly relevant. Although the conjecture stated below arose first in the context of  hypersurfaces with non-vanishing Gaussian curvature, it is nevertheless expected that it should hold for arbitrary sufficiently smooth hypersurfaces.

\begin{conjecture}\emph{\textbf{(Mizohata--Takeuchi)}} For any $C^2$ compact convex hypersurface $\Sigma$ in $\mathbb{R}^n$, the inequality
\begin{equation*}
    \int |\widehat{gd\sigma}|^2w \leq C \|Xw\|_\infty \int |g|^2
\end{equation*}
holds for all $g\in L^2(\Sigma)$ and all weights $w:\mathbb{R}^n\rightarrow [0,+\infty)$, for some $C>0$ that only depends on $\Sigma$. 
\end{conjecture}
Here, $X$ denotes the $X$-ray transform, so that
\begin{equation*}
    \|Xw\|_\infty=\sup_{\ell}\int_\ell w,
\end{equation*}
where the supremum is taken over all lines $\ell$ in $\mathbb{R}^n$. By the compactness of $\Sigma$ and uncertainty principle considerations, the Mizohata--Takeuchi conjecture is equivalent to 
\begin{equation*}
  \int |\widehat{gd\sigma}|^2w \leq C \sup_{T} w(T) \int |g|^2
\end{equation*}
where the supremum is taken over all $1$-neighbourhoods $T$ of doubly-infinite lines in $\mathbb{R}^n$. In particular we may -- and indeed we shall -- assume that $w$ is roughly constant at scale $1$.

The Mizohata--Takeuchi conjecture is open in all dimensions, including $n=2$ (where the Fourier extension conjecture has been resolved).\footnote{It is a nice observation of Bennett and Nakamura~\cite[p.129]{BN21} that when $n=2$, the Mizohata--Takeuchi conjecture implies the Fourier extension conjecture.}. It would directly follow from the truth of the stronger conjecture
\begin{equation}\label{eq: Stein's conjecture}
    \int |\widehat{gd\sigma}|^2 w \leq  \; \; C \int  |g(\xi)|^2\sup_{\ell\parallel N(\xi)}Xw(\ell)\; d\sigma(\xi),
\end{equation}
a formulation of which in the related context of the disc multipliers is due to Stein \cite{St78}; here, $N(\xi)$ denotes the normal to $\Sigma$ at $\xi$.

When $\Sigma=\mathbb{S}^{n-1}$ and the weight is radial, the Mizohata--Takeuchi conjecture is known to hold (see \cite{CRS92, BRV97, CS97a, CS97b, CSV07}), and the Stein-like conjecture in the same setting is a trivial consequence of this. When the weight is constant on parallel hyperplanes and the hypersurface is arbitrary, both conjectures are true. This can be seen by using an affine change of variables to reduce to the case of horizontal hyperplanes and a hypersurface parametrised as $(t, \gamma(t))$ for $t \in \mathbb{R}^{n-1}$, and in this case Plancherel's theorem in $\mathbb{R}^{n-1}$ gives the result directly.  When $\Sigma=\mathbb{S}^1$ and the weight is a measure supported on $\mathbb{S}^1$, both conjectures are also known \cite{BCSV06}. Little is known beyond these three cases.

One way to measure partial progress on the Mizohata--Takeuchi conjecture is to consider inequalities of the form 
\begin{equation*}
    \int_{B_R} |\widehat{gd\sigma}|^2w \leq C R^{\alpha}\|Xw\|_\infty \int |g|^2
\end{equation*}
where $B_R$ is the ball of radius $R$ centred at $0$, and to attempt to establish such inequalities with the exponent $\alpha$ as small as possible. By the Agmon--H\"ormander trace inequality and the local constancy of $w$ at scale $1$ we have 
\begin{equation}\label{eq: easy 2}
    \int_{B_R} |\widehat{gd\sigma}|^2w \leq C R \|w\|_\infty \int |g|^2  \leq C R\|Xw\|_\infty \int |g|^2 
\end{equation}
in all dimensions $n \geq 2$, and it is known that 
\begin{equation}\label{eq: easy 1}
    \int_{B_R} |\widehat{gd\sigma}|^2w \leq C R^{1/2}\|Xw\|_\infty \int |g|^2\;\;\text{ for }n=2.
\end{equation}
The latter inequality can be traced back to works of Bourgain \cite{B94}, Erdo\~gan \cite{E04} and also Carbery and Seeger \cite{CS00} -- see \cite[Section 4]{BBC08} for further details of inequalities which can be found in the literature and which have \eqref{eq: easy 1} as a consequence. We give a more direct proof of this in Section \ref{section: trivial estimates} below. In more recent developments, it is a consequence of the main result in 
Du and Zhang \cite{DZ19} that one may take any $\alpha > (n-1)/n$ (in fact, with the significantly smaller functional $\sup_{x, \, 1 \leq r \leq R} w(B(x,r))/r^{n-1}$ in place of $\|Xw\|_\infty$) for arbitrary $n$.
(See also Shayya \cite{Sh21} and Du et al \cite{DGOWWZ21}, who gave alternative arguments when $n=3$ for $\alpha > 6/7$ and $\alpha > 2/3$ respectively.) In Theorem~\ref{theorem: all weights} below we show that one may take any $\alpha > (n-1)/(n+1)$ in all dimensions. 

See also \cite{BN21, BNS22} for a tomographic approach to the Mizohata--Takeuchi conjecture, \cite{Sh22} for related weighted $L^2\rightarrow L^4$ estimates on the extension operator, and \cite{GWZ22} for variants of the conjecture when the supports of $g$ and $w$ are respectively contained in and equal to neighbourhoods of algebraic varieties.

\textit{\textbf{Notation.}} The control we shall obtain on $\int_{B_R}|\widehat{gd\sigma}|^2w$ will be  accompanied by multiplicative losses of the form $C_\epsilon R^\epsilon$ for any $\epsilon >0$. In order to facilitate expression of this, we adopt the following notation.

For any non-negative quantities $A$ and $B$ (which may depend on $R$), $A\lesssim B$ means that $A\leq c B$ for some constant $c$ that depends only on $\Sigma$ and the ambient dimension. Likewise, $A\gtrsim B$ means that $B\lesssim A$, while $A\sim B$ means that $A\lesssim B$ and $A\gtrsim B$. With $R \geq 1$ fixed, $A\lessapprox B$ means that, for every $\epsilon>0$, there exists a constant $C_\epsilon$, depending only on $\epsilon$, $\Sigma$ and the ambient dimension, such that $A\leq C_\epsilon R^\epsilon B$. Similarly, $A\gtrapprox B$ means that $B\lessapprox A$, while $A\approx B$ means that $A\lessapprox B$ and $A\gtrapprox B$.

For a weight $w$ on $\mathbb{R}^n$ and $A\subset\mathbb{R}^n$, we denote by $w(A)$ the integral $\int_A w$ with respect to Lebesgue measure on $\mathbb{R}^n$.

For $n\geq 2$, an $n$-dimensional ball of radius $r$ will be referred to as an $r$\textit{-ball}. A tube of length $r$ and cross section an $(n-1)$-dimensional ball of radius $r^{1/2}$ will be referred to as an $r^{1/2}$\textit{-tube}. With $R \geq 1$ fixed and $1 \leq r \leq R$, we let $\mathbb{T}_r$ be the set of $r^{1/2}$-tubes intersecting $B_R$.

For a line $\ell$ in $\mathbb{R}^n$ and $g\in L^2(\Sigma)$, we write $\ell\perp {\rm supp}\;g$ if the direction of $\ell$ is parallel to one of the normals to ${\rm supp}\;g \subset \Sigma$.

For a tube $T$ in $\mathbb{R}^n$, we write $T\perp {\rm supp}\;g$ if the central line of $T$ is parallel to one of the normals to ${\rm supp}\;g \subset \Sigma$. \hfill $\blacksquare$

\textbf{Statement of results.} In this paper, we present several $L^2$-weighted inequalities for the Fourier extension operator which are related to the Mizohata--Takeuchi conjecture. To place our results in context, we first observe that the Stein--Tomas inequality 
\begin{equation*}
    \|\widehat{gd\sigma}\|_{L^\frac{2(n+1)}{n-1}(\mathbb{R}^n)}\lesssim \|g\|_2
\end{equation*}
together with H\"older's inequality implies that
\begin{equation*}
\int_{B_R} |\widehat{gd\sigma}|^2w \lesssim \left(\int_{B_R}w^{\frac{n+1}{2}}\right)^{\frac{2}{n+1}} \int |g|^2
\end{equation*}
for all $g$ and all non-negative $w$. The first Mizohata--Takeuchi-type estimates that we present give a significant improvement over this inequality, and follow from the refined Stein--Tomas-type estimate in \cite{GIOW20}. They are given in Theorem \ref{theorem: all weights} below. The main inequality of this result, \eqref{eq: mainineq}, is closely related to, but logically independent from, the Mizohata--Takeuchi conjecture, and it is sharp in the sense we discuss below the statement. Its consequence \eqref{eq: MTmain} is also sharp given the techniques that we employ; see \cite{Gu22}, the remarks at the end of this section and Section~\ref{section: Guth's example}. Estimates which improve on Theorem \ref{theorem: all weights} appear in Lemma \ref{lemma: small caps} (for $g$ with small support), as well as in Theorems \ref{theorem: hor slabs} and \ref{theorem: all slabs} (for weights that are constant on slabs), and arise as consequences of Theorem \ref{theorem: all weights}.

\begin{theorem}\label{theorem: all weights} Let $n\geq 2$. For every $\epsilon>0$, there exists a positive constant $C_\epsilon$, which depends only on $\Sigma$ and $\epsilon$, such that
\begin{eqnarray}\label{eq: mainineq}
\int_{B_R} |\widehat{gd\sigma}|^2w  \leq C_\epsilon R^\epsilon \sup_{T\in\mathbb{T}_R:\;T\perp {\rm supp}\;g} \left(\int _T w^{\frac{n+1}{2}}\right)^{\frac{2}{n+1}}\int |g|^2,
\end{eqnarray}
and in particular 
\begin{align}\label{eq: MTmain}
    \int_{B_R} |\widehat{gd\sigma}|^2 w 
    \lessapprox R^{\frac{n-1}{n+1}}\left(\sup_{\ell\perp {\rm supp}\;g}Xw(\ell)\right)\int |g|^2
\end{align}
for all $R\geq 1$, $g\in L^2(\Sigma)$ and weights $w:\mathbb{R}^n\rightarrow [0,+\infty)$.
\end{theorem}
The second statement follows from the first upon noting that 
\begin{eqnarray*}
\begin{aligned}
\sup_{T\in\mathbb{T}_R:\;T\perp {\rm supp}\;g} \left(\int _T w^{\frac{n+1}{2}}\right)^{\frac{2}{n+1}} 
        & \leq  \|w\|_\infty^{\frac{n-1}{n+1}}\left(\sup_{T\in\mathbb{T}_R:\;T\perp {\rm supp}\;g} w(T)\right)^{\frac{2}{n+1}} \\
        &\lesssim R^{\frac{n-1}{n+1}}\;\|w\|_\infty^{\frac{n-1}{n+1}}\;\left(\sup_{\ell\perp {\rm supp}\;g}Xw(\ell)\right)^{\frac{2}{n+1}}\\
\end{aligned}
\end{eqnarray*}
and using the approximate constancy of $w$ at scale $1$.

Notice that Theorem~\ref{theorem: all weights}, unlike the Mizohata--Takeuchi conjecture itself, requires non-vanishing curvature of $\Sigma$.

\begin{remark}{\rm Inequality \eqref{eq: mainineq} of Theorem~\ref{theorem: all weights} is sharp in the following senses. Firstly, if the exponent $r$ is such that 
\begin{equation*}
\int_{B_R} |\widehat{gd\sigma}|^qw  \lesssim \left(\int_{B_R}w^{r}\right)^{1/r}\left(\int |g|^p\right)^{q/p}
\end{equation*}
(which, by duality, is equivalent to an $L^p-L^{qr'}$ Fourier extension estimate) holds, then necessarily $ 1/qr' \leq (n-1)/(n+1)p'$; 
so the exponent $(n+1)/2$ appearing in \eqref{eq: mainineq} 
(in which $p=q=2$) cannot be increased, irrespective of the size of the tubes $T \subset B_R$. Secondly, fixing $r = (n+1)/2$ in \eqref{eq: mainineq}, we cannot reduce the width of the tubes appearing to be significantly smaller than $R^{1/2}$. These two assertions can both be seen by testing as usual on $g$ the indicator function of an $R^{-1/2}$-cap and $w$ the indicator of the dual $R^{1/2}$-tube. Moreover, we may not take $\epsilon = 0$ in \eqref{eq: mainineq}, and it is likely that
when $n=2$, we may be able to replace the $R^\epsilon$ term by a power
of $\log R$; see Remark~\ref{remark: epsilon} below.}

    
\end{remark}

Theorem \ref{theorem: all weights} will follow from the more precise Theorem \ref{theorem: all weights refined}, in which $\mathbb{T}_R$ is replaced by the set of tubes featuring in the wave packet decomposition of $g$ at scale $R$.

We now turn to our other results. Theorems \ref{theorem: hor slabs} and \ref{theorem: all slabs} below are improvements of Theorem \ref{theorem: all weights} for weights that exhibit a level of local constancy along slabs. In the extreme case where there is no such local constancy beyond on unit scale, both theorems reduce to Theorem \ref{theorem: all weights}. Theorem \ref{theorem: hor slabs} involves slabs that are `roughly parallel' to caps of $\Sigma$, while Theorem \ref{theorem: all slabs} addresses the general case. 

Both theorems (and, in fact, the more precise Theorems \ref{theorem: hor slabs refined} and \ref{theorem: all slabs refined}) will follow from a strengthened version of Theorem \ref{theorem: all weights} for functions $g$ with small support (Lemma \ref{lemma: small caps} below) which we will prove for all weights. 

In order to state Theorems~\ref{theorem: hor slabs} and \ref{theorem: all slabs}, we first establish some further notation, and introduce a quantity which is intermediate between the quantity 
$$\sup_{T\in\mathbb{T}_R:\;T\perp {\rm supp}\;g} \left(\int _T w^{\frac{n+1}{2}}\right)^{\frac{2}{n+1}}$$ 
occuring in Theorem~\ref{theorem: all weights} and a quantity more directly geared towards that occuring in the Mizohata--Takeuchi conjecture itself. This will involve considering an amalgam of `running averages' of $w$ at certain scales related to the level of constancy that we are assuming, which is measured by a parameter $1 \leq \rho \leq R$ which we now fix. Let $E \subset \Sigma$. For each $T_R \in \mathbb{T}_R$ such that $T_R \perp E$, we cover $T_R$ by essentially disjoint tubes $S_\rho \in \mathbb{T}_\rho$ which are parallel to and contained in $T_R$. For $w:\mathbb{R}^n\rightarrow [0,+\infty)$ and $E \subset \Sigma$ we define 
\begin{equation*}
    A_{\rho,R,E}(w):= \frac{1}{\rho^{\frac{n-1}{2}}} \sup_{T_R\in \mathbb{T}_R:\; T_R\perp E} \left(\sum_{S_\rho \subset T_R} w(S_\rho)^{\frac{n+1}{2}}\right)^{\frac{2}{n+1}},
    \end{equation*}
a quantity which can be expressed more geometrically as 
\begin{equation*}
   \sup_{T_R\in \mathbb{T}_R:\; T_R\perp E} \left(\sum_{S_\rho \subset T_R} \left(\frac{w(S_\rho)}{|S_\rho|}\right)^{\frac{n+1}{2}}|S_\rho|\right)^{\frac{2}{n+1}}
    \end{equation*}
and thus is seen to increase as $\rho$ gets smaller.\footnote{By H\"older's inequality we have, for $\lambda \geq 1$ and a tessellation of an $S_{\lambda \rho}$ by $S_\rho$'s,
\begin{equation*}
\left(\frac{w(S_{\lambda\rho})}{|S_{\lambda\rho}|}\right)^{\frac{n+1}{2}}|S_{\lambda\rho}| \lesssim \sum_{S_\rho \subset S_{\lambda \rho}} \left(\frac{w(S_{\rho})}{|S_{\rho}|}\right)^{\frac{n+1}{2}}|S_{\rho}|.
\end{equation*}} For $\rho=1$, 
\begin{equation*}
    A_{1,R,E}(w)\sim\sup_{T_R\in \mathbb{T}_R:\; T_R\perp E} \left(\int_{T_R} w^{\frac{n+1}{2}}\right)^{\frac{2}{n+1}}
\end{equation*}
is the quantity appearing on the right-hand side of Theorem \ref{theorem: all weights}, controlling the $L^2(E)\rightarrow L^2(w)$-norm of the extension operator. Theorem \ref{theorem: all weights} fails in general for $g$ supported on $E$ if the above quantity is replaced by the smaller
\begin{equation*}
    A_{R,R,E}(w)=\sup_{T_R\in \mathbb{T}_R:\; T_R\perp E}\; \frac{w(T_R)}{R^{\frac{n-1}{2}}}
\end{equation*}
(and in fact by $A_{\rho,R,E}(w)$ for any $\rho \gg1$, as can be seen by taking $g$ to be the indicator function of a $1$-cap and $w$ the indicator function of the unit ball). In the results which follow, however, we shall show that under certain auxiliary conditions ($g$ being supported on a small cap, or the weight being the indicator function of a union of small slabs), Theorem \ref{theorem: all weights} nevertheless does hold for $g\in L^2(E)$ if we replace the quantity $A_{1,R,E}(w)$ with $A_{\rho,R,E}(w) $ for an appropriate choice of $\rho$. To further compare these two quantities, observe that
\begin{equation}\label{eq: comparing}
    A_{\rho,R,E}(w) \leq \sup_{S_\rho\in\mathbb{T}_\rho:\;S_\rho\perp E}\left(\frac{w(S_\rho)}{|S_\rho|}\right)^{\frac{n-1}{n+1}}\sup_{T_R \in \mathbb{T}_R, T_R \perp E} w(T_R)^{\frac{2}{n+1}},
\end{equation}
which becomes
\begin{equation*}
    A_{\rho,R,E}(w) \leq \sup_{S_\rho\in\mathbb{T}_\rho:\;S_\rho\perp E}\left(\frac{w(S_\rho)}{|S_\rho|}\right)^{\frac{n-1}{n+1}}\; A_{1,R,E}(w)
\end{equation*}
when $w$ is an indicator function (which we may well assume for our purposes).

In situations in which we are able to  bound the $L^2(E)\rightarrow L^2(w)$-norm of the extension operator by $A_{\rho,R,E}(w)$, inequality \eqref{eq: comparing} leads to improved bounds in terms of $\|Xw\|_\infty$; in particular, to a gain on Theorem \ref{theorem: all weights} by a factor $\rho^{-\frac{n-1}{n+1}}$. Indeed, by \eqref{eq: comparing},
\begin{equation*}
A_{\rho,R,E}(w) \leq \left(\frac{\|Xw\|_\infty}{\rho}\right)^{\frac{n-1}{n+1}} (R^{\frac{n-1}{2}}\|Xw\|_\infty)^{\frac{2}{n+1}}\lesssim \left(\frac{R}{\rho}\right)^{\frac{n-1}{n+1}}\sup_{\ell\perp E} Xw(\ell).
\end{equation*}
A situation such as this arises when $g$ is supported in a $\rho^{-1/2}$-cap of $\Sigma$ (that is, the intersection of $\Sigma$ with a $\rho^{-1/2}$-ball), and is summarised in Lemma \ref{lemma: small caps} below. The lemma will in turn be used in conjunction with a decoupling argument to derive Theorems~\ref{theorem: hor slabs} and \ref{theorem: all slabs} for all functions $g$ and restricted classes of weights. Note that, in Lemma \ref{lemma: small caps} below, the subscript $\tau$ on $g_\tau$ is not strictly needed, but we retain it to emphasise its support.
\begin{lemma}\textbf{\emph{(Small caps)}}\label{lemma: small caps} For every $\epsilon>0$, there exists $C_\epsilon>0$ such that
for all weights $w:\mathbb{R}^n\rightarrow [0,+\infty)$, whenever $1\leq\rho\leq R$, $\tau$ is a $\rho^{-1/2}$-cap of $\Sigma$ and $g_\tau\in L^2(B^{n-1})$ is supported in $\tau$, we have
\begin{equation*}
    \int_{B_R}|\widehat{g_\tau d\sigma}|^2 w \leq C_\epsilon R^\epsilon\;  A_{\rho,R, \;{\rm supp}\;g_\tau}(w) \int |g_\tau|^2,
\end{equation*}
and therefore also
\begin{equation}\label{eq:MTsmallcaps}
    \int_{B_R}|\widehat{g_\tau d\sigma}|^2 w \lessapprox\; \left(\frac{R}{\rho}\right)^{\frac{n-1}{n+1}}\sup_{\ell\perp {\rm supp}\;g_\tau} Xw(\ell)  \int |g_\tau|^2.
\end{equation}
\end{lemma}
In order to state Theorems~\ref{theorem: hor slabs} and \ref{theorem: all slabs}, we need to make precise what we mean by a slab, and by 
a slab being `roughly parallel' to caps of $\Sigma$.

\begin{definition}
{\rm Fix $R\geq 1$, $1\leq\rho\leq R$ and $0\leq \nu\leq \pi/2$. We define a $\rho^{1/2}$\textit{-slab} to be any affine copy of the 1-neighbourhood of an $(n-1)$-dimensional $\rho^{1/2}$-ball in $\mathbb{R}^n$. We say that a slab is $\nu$\textit{-parallel to} $\Sigma$ if all normals to $\Sigma$ create angle at least $\nu$ with the slab (that is, they create angle at most $\frac{\pi}{2}-\nu$ with the normal to the slab).}
\end{definition}

In this definition, $\nu$ is a measure of how large the angles are between the slab and the normals to $\Sigma$. The larger $\nu$ is, the larger these angles are, and the more `parallel' $\Sigma$ and the slab look.

With these preliminaries in hand, we are now ready to state our remaining results. In the first two results which follow, the implicit constant blows up as $\nu \downarrow 0$. Thus, the interesting cases of these two results are those in which $\nu$ is large, i.e. when the slabs create large angles with the normals to $\Sigma$. If for instance $\Sigma$ is roughly horizontal (i.e. all normals to $\Sigma$ are within angle $\leq 1/100$ from the vertical direction), then Theorem \ref{theorem: hor slabs} gives meaningful results for slabs that are also nearly horizontal (e.g. creating angle $\geq 2/100$ with the vertical direction).

\begin{theorem}\textbf{\emph{(Slabs $\nu$-parallel to $\Sigma$)}}\label{theorem: hor slabs} For every $0 < \nu \leq \pi/2$ and $\epsilon>0$, there exists $C_{\epsilon,\nu}>0$ such that the following hold. Let $g\in L^2(\Sigma)$. For $R\geq 1$ and $R^\epsilon\lesssim_\epsilon\rho\leq R$, let $w:\mathbb{R}^n\rightarrow [0,+\infty)$ be a weight of the form $\sum_{s\in\mathcal{S}}c_s\chi_s$, where $\mathcal{S}$ is a set of disjoint $\rho^{1/2}$-slabs $\nu$-parallel to $\Sigma$. Then the inequality
\begin{align*}
        \int_{B_R} |\widehat{gd\sigma}|^2w &\leq C_{\epsilon,\nu} R^\epsilon A_{\rho, R, \;{\rm supp}\;g} (w) \int |g|^2\\
        & \lessapprox_\nu \left(\frac{R}{\rho}\right)^{\frac{n-1}{n+1}}\sup_{\ell\perp {\rm supp}\;g}Xw(\ell)\int |g|^2
    \end{align*}
holds. In fact, if
\begin{equation*}
    g=\sum_{\tau\in\mathfrak{T}}g_\tau,\qquad{{\rm supp}\;g_\tau\subset\tau}
\end{equation*}
for some boundedly overlapping family $\mathfrak{T}$ of $\rho^{-1/2}$-caps $\tau$ of $\Sigma$, then
\begin{equation*}
    \int_{B_R} |\widehat{gd\sigma}|^2w\lessapprox_\nu 
    \sum_{\tau\in\mathfrak{T}}A_{\rho, R, \;{\rm supp}\; g_\tau}(w)\int |g_\tau|^2
    \lessapprox_\nu \left(\frac{R}{\rho}\right)^{\frac{n-1}{n+1}} \sum_{\tau\in\mathfrak{T}}\;\sup_{\ell\perp {\rm supp}\;g_\tau}Xw(\ell)\int |g_\tau|^2.
\end{equation*}
\end{theorem}

It follows that Stein's stronger conjecture \eqref{eq: Stein's conjecture} (and thus the Mizohata--Takeuchi conjecture) holds under the conditions of Theorem \ref{theorem: hor slabs} when the slabs involved are $R^{1/2}$-slabs. We single this out explicitly as a corollary.

\begin{corollary}\label{corollary: hor slabs}
Let $R \geq 1$ and suppose that $w$ is a weight of the form $\sum_{s\in\mathcal{S}}c_s\chi_s$, where $\mathcal{S}$ is a set of disjoint $R^{1/2}$-slabs which are $\nu$-parallel to $\Sigma$ for some $0 < \nu \leq \pi/2$. Then 
\begin{equation*}
    \int_{B_R} |\widehat{gd\sigma}|^2w
      \lessapprox_\nu \int |g(\xi)|^2\sup_{\ell\parallel N(\xi)} Xw(\ell) d\sigma(\xi).
\end{equation*}
for all $g\in L^2(\Sigma)$. 
\end{corollary}

Stein's conjecture continues to hold even when the slabs are curved. The precise formulation of this appears in Corollary \ref{corollary: flakes}, and it is proved using a direct method, which does not rely on Theorem~\ref{theorem: all weights}, and which also featured in \cite{Gu22}.

A substitute result for Theorem \ref{theorem: hor slabs} in the case where there is no restriction on $\nu$ (i.e. when the slabs can create arbitrarily small angles with normals to $\Sigma$) is as follows. 

\begin{theorem}\textbf{\emph{(All slabs)}}\label{theorem: all slabs} For every $\epsilon>0$, there exists $C_{\epsilon}>0$ such that the following hold. Let $g\in L^2(\Sigma)$. For $R\geq 1$ and $R^\epsilon\lesssim_\epsilon\rho\leq R$, let $w:\mathbb{R}^n\rightarrow [0,+\infty)$ be a weight of the form $\sum_{s\in\mathcal{S}}c_s\chi_s$, where $\mathcal{S}$ is a set of disjoint $\rho^{1/2}$-slabs with no conditions on their directions. Then the inequality
\begin{align*}
        \int_{B_R} |\widehat{gd\sigma}|^2w &\leq C_{\epsilon} R^\epsilon A_{\rho^{1/2}, R, \; {\rm supp}\;g} (w) \int |g|^2\\
        & \lessapprox \left(\frac{R}{\rho^{1/2}}\right)^{\frac{n-1}{n+1}}\sup_{\ell\perp {\rm supp}\;g}Xw(\ell)\int |g|^2
    \end{align*}
holds. In fact, if
\begin{equation*}
    g=\sum_{\tau\in\mathfrak{T}}g_\tau,\qquad{{\rm supp}\;g_\tau\subset\tau}
\end{equation*}
for some boundedly overlapping family $\mathfrak{T}$ of $\rho^{-1/4}$-caps $\tau$ of $\Sigma$, then
\begin{equation*}
    \int_{B_R} |\widehat{gd\sigma}|^2w\lessapprox 
    \sum_{\tau\in\mathfrak{T}}A_{\rho^{1/2}, R, \; {\rm supp}\; g_\tau}(w)\int |g_\tau|^2
    \lessapprox \left(\frac{R}{\rho^{1/2}}\right)^{\frac{n-1}{n+1}} \sum_{\tau\in\mathfrak{T}}\;\sup_{\ell\perp {\rm supp}\;g_\tau}Xw(\ell)\int |g_\tau|^2.
\end{equation*}
\end{theorem}

\vspace{0.1in}

\begin{corollary}\textbf{\emph{($R^{1/2}$-slabs)}}
 Let $R \geq 1$ and suppose that $w$ is a weight of the form $\sum_{s\in\mathcal{S}}c_s\chi_s$, where $\mathcal{S}$ is a set of disjoint $R^{1/2}$-slabs. Then 
\begin{equation*}
    \int_{B_R} |\widehat{gd\sigma}|^2w \lessapprox R^{\frac{n-1}{2(n+1)}}\sup_{\ell\perp {\rm supp}\;g}Xw(\ell)\int |g|^2.
\end{equation*}
for all $g\in L^2(\Sigma)$. 
\end{corollary}

\vspace{0.2in}

\textbf{Sharpness of inequality \eqref{eq: MTmain} given the choice of technique.} During the recent talk \cite{Gu22}, which in fact partially inspired the work in this paper, Guth explained that, using only basic local constancy and local $L^2$-orthogonality properties of the functions $\widehat{gd\sigma}$ -- which are indeed the only properties that we exploit in proving Theorem \ref{theorem: all weights} -- one {\em cannot} prove the Mizohata--Takeuchi conjecture for $B_R$ with a loss better than $\sim  (\log R)^{-3} R^{\frac{n-1}{n+1}}$. 

This means that inequality \eqref{eq: MTmain} of Theorem \ref{theorem: all weights}, which establishes the conjecture with a loss of $\lessapprox R^{\frac{n-1}{n+1}}$, is essentially sharp given the techniques used.

Guth's argument is discussed in Section \ref{section: Guth's example} for purposes of self-containment.

\vspace{0.2in}

\textbf{Acknowledgements.} We would like to thank Larry Guth, whose inspiring talk \cite{Gu22} partially motivated the work in this paper, for giving us permission to present here a version of his main argument from that talk. We also thank Jonathan Bennett for many illuminating conversations on this topic. The first author would like to acknowledge support from a Leverhulme Fellowship while part of this research was undertaken, and to thank David Beltr\'an and Bassam Shayya for some helpful conversations. The third author would like to acknowledge support from NSF Grant DMS-2238818 and  DMS-2055544. 

\section{Preliminaries}\label{section: preliminaries}

For our purposes, we may assume that all normals to $\Sigma$ have angle at most $1/100$ from the vertical direction, and that the projection of $\Sigma$ on the hyperplane $\mathbb{R}^{n-1}\times\{0\}$ is contained in the unit ball $B^{n-1}$ centred at 0. This convention allows us to assume that $\Sigma$ has a parametrisation
\begin{equation*}
    \Sigma=\{\Sigma(\omega):=(\omega, h(\omega))\text{, for }\omega\in B^{n-1}\}
\end{equation*}
for some $h:B^{n-1}\rightarrow\mathbb{R}$, and to work with the operator $E$ instead of $\widehat{\cdot\;  d\sigma}$, where
\begin{equation*}
    Eg(x):=\int_{B^{n-1}} e^{2\pi i \langle x,\Sigma(\omega)\rangle}g(\omega)d\omega \; \text{ for }x\in\mathbb{R}^n.
\end{equation*}
From now on, for fixed $\Sigma$ and$\epsilon>0$, we say that a quantity $C(R,\epsilon)$ satisfies
\begin{equation*}
    C(R,\epsilon)={\rm RapDec}_\epsilon(R)
\end{equation*}
if for every $N\in\mathbb{N}$ there exists a non-negative constant $C_{N,\epsilon}$ such that uniformly in $R \geq 1$ we have
\begin{equation*}
    |C(R,\epsilon)|\leq C_{N,\epsilon}R^{-N}.
\end{equation*}

\textit{\textbf{Wave packet decomposition adapted to }}$\boldsymbol{B_R}$. Let $\epsilon>0$ and $0 < \delta\ll \epsilon$. Fix $R\gg 1$, and cover $B^{n-1}$ by boundedly overlapping balls $\theta$ of radius $R^{-1/2}$. The set of these balls will be denoted by $\Theta_R$, and the balls will be referred to as $R^{-1/2}$\textit{-caps}. Let $\{\psi_\theta\}_{\theta\in\Theta_R}$ be a smooth partition of unity adapted to this cover. Thus,
\begin{equation*}
    g=\sum_{\theta\in \Theta_R} \psi_\theta g
\end{equation*}
for any $g:\mathbb{R}^{n-1}\rightarrow\mathbb{C}$ supported in $B^{n-1}$ (and belonging to some suitable class). Now, cover $\mathbb{R}^{n-1}$ by boundedly overlapping balls of radius $CR^{(1+\delta)/2}$ and centres on the lattice $V_R:=R^{(1+\delta)/2}\mathbb{Z}^{n-1}$. There exists a bump function $\eta$, adapted to the ball $B(0, R^{(1+\delta)/2})$, so that the bump functions $\eta_v:=\eta(\cdot -v)$, over $v\in V_R$, form a partition of unity for this cover. It follows that, with $\widehat{\cdot}$ and $\widecheck{\cdot}$ denoting the $(n-1)$-dimensional Fourier transform and its inverse respectively,
\begin{equation*}
    \widecheck{g}=\sum_{(\theta,v)}\eta_v (\psi_\theta g)^{\widecheck{}}
\end{equation*}
and thus
\begin{equation*}
    g=\sum_{(\theta,v)}\widehat{\eta_v}\ast (\psi_\theta g)
\end{equation*}
for all $g$ as above. Finally, restrict each of the above summands to the corresponding cap $\theta$. In particular, let
\begin{equation*}
    g_{\theta,v}:=\widetilde{\psi}_\theta \cdot (\widehat{\eta_v}\ast (\psi_\theta g)),
\end{equation*}
where $\widetilde{\psi}_\theta:=\widetilde{\psi}(R^{1/2}(\cdot-\omega_\theta))$ for some fixed smooth bump function $\widetilde{\psi}$ (where $\omega_\theta$ is the centre of the cap $\theta$), chosen so that $\widetilde{\psi}_\theta$ is supported in $\theta$ and equals 1 on the $cR^{1/2}$-neighbourhood of ${\rm supp}\;\psi_\theta$, for some small $c>0$.

The $g_{\theta,v}$ are the \textit{wave packets of $g$ at scale $R$}, while $\{g_{\theta,v}\}_{(\theta,v)\in\Theta_R\times V_R}$ constitutes the \textit{wave packet decomposition of $g$} at this scale. Note that the decomposition is $\epsilon$-dependent.

The function $g$ is roughly the sum of its wave packets, all of which are roughly orthogonal. More precisely, note that the function $\widehat{\eta_v}$ is rapidly decaying when $|\omega|\gg R^{-(1+\delta)/2}$, so
\begin{equation*}
    \|g_{\theta,v}-\widehat{\eta_v}\ast (\psi_\theta g)\|_\infty\leq {\rm RapDec}_\epsilon(R)\|g\|_2\text{ for each }(\theta,v),
\end{equation*}
hence
\begin{equation}\label{eq:wp1}
    \|g-\sum_{(\theta,v)\in \Theta_R\times V_R}g_{\theta,v}\|_\infty \leq {\rm RapDec}_\epsilon(R)\|g\|_2.\tag{wp1}
\end{equation}
The functions $g_{\theta,v}$ are almost orthogonal, in the sense that
\begin{equation}\label{eq:wp2}
    \|\sum_{(\theta,v)\in \mathbb{W}} g_{\theta,v}\|_2^2 \sim \sum_{(\theta,v)\in\mathbb{W}}\|g_{\theta,v}\|_2^2\tag{wp2}
\end{equation}
for every subset $\mathbb{W}$ of $\Theta_R\times V_R$.

It turns out that, for every $(\theta,v)$, $Eg_{\theta,v}$ is essentially supported in 
\begin{equation*}
    T_{\theta,v}:=\big\{x\in B_R: |x'+x_n \partial_\omega h(\omega_\theta)-v|\leq R^{1/2+\delta}\big\},
\end{equation*}
the $R^{1/2+\delta}$-tube in $B_R$ whose central line passes through $(v,0)$ and has direction the normal $N(\theta):=(\partial_\omega h (\omega_\theta),-1)$ to the cap $\Sigma(\theta)$. Indeed, it follows by a non-stationary phase argument that
\begin{equation}\label{eq:wp3}
    |Eg_{\theta,v}(x)|\leq (1+R^{-1/2}|x'+x_n\partial_\omega h(\omega_\theta)-v|)^{-(n+1)}{\rm RapDec}_\epsilon(R)\|g\|_2,\;\;\;\forall x\in B_R\setminus T_{\theta,v};\tag{wp3}
\end{equation}
a detailed analysis can be found in \cite{Gu18}.

Due to the curvature of $\Sigma$, different surface caps $\Sigma(\theta)$ have different normals, so there is a one-to-one correspondence between the pairs $(\theta,v)$ and the tubes $T_{\theta,v}$. We may thus denote each wave packet $g_{\theta,v}$ by $g_T$, for the tube $T=T_{\theta,v}$. 

Henceforth, denote
\begin{equation*}
    \mathbb{T}_\epsilon(B_R):=\{T_{\theta,v}:\;(\theta,v)\in\Theta_R\times V_R\text{ and }T_{\theta,v}\cap B_R\neq\varnothing\}
\end{equation*}
and
\begin{equation*}
    \mathbb{T}^\theta_{\epsilon}(B_R):=\{T_{\overline{\theta},v}:|\overline{\theta}-\theta|\lesssim R^{-1/2}, \;v\in V_R\text{ and }T_{\theta,v}\cap B_R\neq\varnothing\}
\end{equation*}
for each $\theta\in \Theta_R$, where the implicit multiplicative constant is sufficiently large. The above analysis ensures that
\begin{equation}\label{eq:wp4}
    Eg(x)=\sum_{T\in \mathbb{T}_\epsilon(B_R)}Eg_T(x)+{\rm RapDec}_\epsilon(R)\int |g|^2\;\;\text{ for all }x\in B_R,\tag{wp4}
\end{equation}
while also that any function $g_\theta$ supported on $\theta\in \Theta_R$ satisfies
\begin{equation}\label{eq:wp5}
    Eg_\theta(x)=\sum_{T\in \mathbb{T}_\epsilon^\theta(B_R)}Eg_T(x)+{\rm RapDec}_\epsilon(R)\int |g|^2\;\;\text{ for all }x\in B_R.\tag{wp5}
\end{equation}
We will be referring to $\{g_T\}_{T\in\mathbb{T}_\epsilon(B_R)}$ as the \textit{wave packet decomposition} of $g$ \textit{adapted to} $B_R$.

\textit{\textbf{Wave packet decompositions adapted to other balls.}} Let $R^\epsilon\lesssim_\epsilon\rho\leq R$, and fix a ball $B=B(y,\rho)$. For $x\in\mathbb{R}^n$, set $\widetilde{x}:=x-y$. It holds that
\begin{eqnarray*}
\begin{aligned}
Eg(x)&=\int e^{2\pi i \langle x, \Sigma(\omega)\rangle}g(\omega)d\omega\\
    &=\int e^{2\pi i \langle \widetilde{x},\Sigma(\omega)\rangle} e^{2\pi i \langle y, \Sigma(\omega)\rangle}g(\omega)d\omega\\
    &=E\widetilde{g}(\widetilde{x}),
\end{aligned}
\end{eqnarray*}
where $\widetilde{g}(\omega)=e^{2\pi i \langle y, \Sigma(\omega)\rangle}g(\omega)$. For every $x\in B$, $\widetilde{x}$ lives in $B_\rho$; therefore, by the earlier discussion,
\begin{align}\label{eq:wp6}
    Eg(x)&=\sum_{T\in\mathbb{T}_\epsilon(B_\rho)}E\widetilde{g}_T(\widetilde{x})+{\rm RapDec}_\epsilon(\rho)\int |\widetilde{g}|^2 \tag*{ }\\
        &=\sum_{T\in\mathbb{T}_\epsilon(B_\rho)}E\widetilde{g}_T(x-y)+{\rm RapDec}_\epsilon(R)\int |g|^2\qquad{\text{for all }}x\in B.\tag{wp6}
\end{align}
From now on, we will be referring to $\{\widetilde{g}_T\}_{T\in\mathbb{T}_\epsilon(B_\rho)}$ as the \textit{wave packet decomposition} of $g$ \textit{adapted to} $B$. Note that this decomposition is $y$-dependent.

By the above analysis, for every $\rho^{-1/2}$-cap $\tau$ we have
\begin{equation}\label{eq:wp7}
    Eg_\tau(x)=\sum_{T\in\mathbb{T}^\tau_\epsilon(B_\rho)}E\widetilde{g}_T(x-y)+{\rm RapDec}_\epsilon(R)\int |g|^2\qquad{\text{for all }}x\in B.\tag{wp7}
\end{equation}
Each of the wave packets in the above summand is essentially constant in magnitude; this is made rigorous in the subsection below.

\textit{\textbf{Fourier localisation and local constancy.}} Let $\epsilon>0$ and $R^\epsilon \lesssim_\epsilon \rho \leq R$. Fix $g\in L^2(B^{n-1})$ and a $\rho^{-1/2}$-cap $\tau$.

Roughly speaking, since $g_\tau$ is supported in $\tau$, the Fourier transform of $Eg_\tau$ is supported in the $\rho^{-1}$-neighbourghood of $\Sigma(\tau)$. The uncertainty principle then dictates that $|Eg_\tau|$ is essentially constant on each dual object, i.e. on each $\rho^{1/2}$-tube pointing in the direction the normal to $\Sigma(\tau)$.

The above heuristic is made rigorous as follows. Let $\omega(\tau)$ be the centre of $\tau$. The patch of the tangent space to $\Sigma$ at $\Sigma(\omega_\tau)$ that lives over $\tau$ is the set \begin{equation*}
    T_\tau\Sigma:=\Big\{\Sigma(\omega_\tau)+M_\tau\cdot(\omega-\omega_\tau):\;\omega\in \tau\Big\},\;\;\text{where }M_{\tau}:=
    \begin{bmatrix}
    I_{n-1} & 0\\
    \partial_\omega h(\omega_\tau) & 1
    \end{bmatrix}.
\end{equation*}
The convex set
\begin{equation*}
    S(\tau):=\left\{\Sigma(\omega_\tau) + M_\tau \cdot (\omega-\omega_\tau)+t\cdot \frac{N(\omega_\tau)}{\|N(\omega_\tau)\|}:\; \;\omega\in\tau, \;t\in \big[-\rho^{-1},\rho^{-1}\big]\right\}
\end{equation*}
is a `thickening' of the above tangent patch by $\rho^{-1}$ in the direction normal to $\Sigma(\tau)$. The Fourier transform of ${Eg_{\tau}}_{|_{B_R}}$ is essentially supported in a dilation of $S(\tau)$. We are interested in a precise version of this for appropriate cut-offs of $Eg_\tau$. 

In particular, let $\zeta:\mathbb{R}^n\rightarrow\mathbb{R}$ with $\zeta=1$ on $B_1$ and $\zeta=0$ outside $B_{2}$. For every ball $B=B(\overline{x},\rho)$ in $\mathbb{R}^n$ define
\begin{equation*}
    \zeta_B(x):=\zeta\Big(\frac{x-\overline{x}}{\rho}\Big).
\end{equation*}
There exists a constant $C$, depending only on the dimension $n$, such that the following holds.

\begin{proposition}\emph{\textbf{(Fourier localisation)}}\label{prop: Fourier localisation}
Let $R^\epsilon\lesssim_\epsilon\rho\leq R$, and let $g_\tau$ be supported in a $\rho^{-1/2}$-cap $\tau$. Then, for every $\rho$-ball $B$ in $\mathbb{R}^n$,
    \begin{equation*}
        Eg_\tau\cdot \zeta_B=G_\tau+{\rm RapDec}_\epsilon(\rho)\|g_\tau\|_2,
    \end{equation*}
for some $G_\tau:\mathbb{R}^n\rightarrow\mathbb{C}$ with the property that $\widehat{G_\tau}$ is supported in $S(C\cdot\tau)$.
\end{proposition}

The set $C\cdot\tau$ is the $C\rho^{-1/2}$-cap with the same centre as $\tau$. The proof of Proposition \ref{prop: Fourier localisation} is exposed in full detail in \cite{HI22}.

When a function $f$ is Fourier localised on a convex set (such as the slab $S(\tau)$), then to some extent it can be treated as a constant function on objects dual to that convex set. The precise statement appears in Lemmas 6.1 and 6.2 in \cite{GWZ20}. For our purposes, we only need the following corollary.

\begin{proposition}\emph{\textbf{(Local constancy)}}\label{prop: uncertainty principle}
Let $R^\epsilon\lesssim_\epsilon\rho\leq R$. Let $\tau$ be a $\rho^{-1/2}$-cap, and consider a function $f:\mathbb{R}^n\rightarrow\mathbb{C}$ with $\widehat{f}\subset S(\tau)$. Then, every tube $T$ in $\mathbb{R}^n$ with direction $N(\tau)$, radius $\rho^{1/2}$ and length $\rho$ satisfies
\begin{equation*}
    \sup_{x\in T}|f(x)|^2\lesssim \frac{1}{|T|}\int |f|^2w_T,
\end{equation*}
for some non-negative function $\omega_T:\mathbb{R}^n\rightarrow\mathbb{R}$, with $\omega_T=1$ on $T$ and $\omega(x)\sim C_N(1+n(x,T))^{-N}$ for all $x\in\mathbb{R}^n$ and $N\in\mathbb{N}$, where $n(x,T)$ is the smallest $n\in\mathbb{N}$ such that $x\in nT$. In particular, if $g\in L^2(B^{n-1})$ and $B$ is a $\rho$-ball intersecting $T$, then
\begin{equation*}
    \sup_{x\in T}|Eg_{\tau}(x)|^2\lesssim \rho^\delta \frac{1}{|2\widetilde{T}|}\int_{2\widetilde{T}}|Eg_{\tau}|^2+{\rm RapDec}_\epsilon (R)\int |g_\tau|^2.
\end{equation*}
for all $\widetilde{T}$ in $\mathbb{T}_\epsilon^\tau(B)$ intersecting $T$.
\end{proposition}

\begin{proof}
The first conclusion is a direct application of Lemmas 6.1 and 6.2 in \cite{GWZ20}. We now in turn apply this conclusion to the function $Eg_{\tau}\cdot \zeta_B$, which is essentially Fourier supported in $S(C\cdot\tau)$ by Proposition \ref{prop: Fourier localisation}. Respecting the notation of Proposition \ref{prop: Fourier localisation}, denote by $T_C$ the tube with the same central line as $T$, radius $(C^{-2}\rho)^{1/2}$ and length $C^{-2}\rho$. We obtain
\begin{equation*}
    \sup_{x\in T}|Eg_{\tau}(x)|^2=\sup_{x\in T}|Eg_\tau(x)\cdot \zeta_B(x)|^2\lesssim \frac{1}{|T_C|}\int |Eg_\tau\cdot \zeta_B|^2w_{T_C}+{\rm RapDec}_\epsilon(\rho)\int |g_\tau|^2.
\end{equation*}
Since $w_T(x)\sim w_{T_C}(x)$ for all $x\in\mathbb{R}^n$, it holds that
\begin{eqnarray*}
\begin{aligned}
\frac{1}{|T_C|}\int |Eg_\tau\cdot \zeta_B|^2w_{T_C}&\lesssim \frac{1}{|T|}\int |Eg_\tau\cdot\zeta_B|^2w_T\\
    &\lesssim \frac{1}{|T|}\int_{2B} |Eg_\tau|^2w_T\\
    &\sim \frac{1}{|T|}\int_{2B\cap 2\widetilde{T}} |Eg_\tau|^2w_T
    +
    \frac{1}{|T|}\int_{2B\setminus 2\widetilde{T}} |Eg_\tau|^2w_T\\
    &\lesssim \frac{\rho^\delta}{|2\widetilde{T}|}\int_{2\widetilde{T}} |Eg_\tau|^2w_T
    +
    {\rm RapDec}_\epsilon(\rho)\frac{1}{|T|}\int_{2B\setminus 2\widetilde{T}} |Eg_\tau|^2w_T.
\end{aligned}
\end{eqnarray*}
The result follows as, due to the decay properties of $w_T$,
\begin{equation*}
    {\rm RapDec}_\epsilon(\rho)\frac{1}{|T|}\int_{2B\setminus 2\widetilde{T}} |Eg_\tau|^2w_T={\rm RapDec}_\epsilon(\rho)={\rm RapDec}_\epsilon (R)\int |g_\tau|^2.
\end{equation*}
\end{proof}


\section{Some new cases where Mizohata--Takeuchi holds.}\label{section: trivial estimates}

In this section, $\Sigma:=\{(\omega, h(\omega)):\omega\in B^{n-1}\}$ is a fixed hypersurface in $\mathbb{R}^n$, all of whose normals point within angle $1/100$ from the vertical direction. There is no requirement that $\Sigma$ have non-vanishing Gaussian curvature.

The truth of the Mizohata--Takeuchi conjecture for some simple weights (such as indicator functions of neighbourhoods of roughly horizontal hyperplanes or hypersurfaces) implies that the conjecture holds for more complicated weights (superpositions of appropriately large patches of such surfaces). For instance, the Mizohata--Takeuchi conjecture holds for nearly horizontal $R^{1/2}$-slabs (case $\rho=R$ of Theorem \ref{theorem: hor slabs}) because it holds for horizontal hyperplanes (Plancherel).

\begin{definition}
{\rm A $\rho$\textit{-flake} (or simply a \textit{flake}) in $\mathbb{R}^n$ is the 1-neighbourhood of any hypersurface of the form $\{(\omega,\Gamma(\omega)):\omega\in B^{n-1}_\rho\}$, where $B^{n-1}_\rho$ is a $\rho$-ball in $\mathbb{R}^{n-1}$ and $\Gamma:B^{n-1}_\rho\rightarrow\mathbb{R}$. A flake is \textit{nearly horizontal} if all its tangent spaces create angle larger than $2/100$ with the vertical direction.} 
\end{definition}

Note that $\rho$-slabs are $\rho$-flakes. We will usually be taking $\rho \geq 1$. We emphasise that $\Gamma$ and $h$ are unrelated.

Every line normal to $\Sigma$ which intersects a nearly horizontal flake will do so along a line segment of length about $1$.  Therefore, the following lemma states that the Mizohata--Takeuchi conjecture holds when the weight is the indicator of a single nearly horizontal flake. 

\begin{lemma}\label{lemma: single flake} Let $\gamma$ be a nearly horizontal flake in $\mathbb{R}^n$. Then, for all $R\geq 1$ and $g\in L^2(B^{n-1})$,
\begin{equation*}
    \int_{B_R\cap\gamma}|Eg|^2\lessapprox \int |g|^2.
\end{equation*}
\end{lemma}

\begin{proof}
The proof easily follows by induction on scales, and only a sketch is provided here. In particular, the estimate trivially holds when $R\lesssim 1$. For arbitrary larger $R$, we cover the flake $\gamma$ by finitely overlapping $R^{1/2}$-balls $B$. For every one of these balls $B$, we may assume that
\begin{equation*}
    \int_{B\cap\gamma}|Eg|^2\lessapprox \int |g_B|^2,
\end{equation*}
where $g_B$ is the sum of the wave packets $g_T$ of $g$ at scale $R$ that intersect $B$. The functions $g_B$ are essentially orthogonal, as each of the tubes $T$ in question has width $R^{1/2+\delta}$ (where as in Section~\ref{section: preliminaries}, $0< \delta \ll \epsilon$) and creates angle $\gtrsim 1$ with the flake, hence it intersects $R^{O(\delta)}$ of the balls $B$. Adding up the above estimate over all $B$ completes the proof. \end{proof}

\begin{remark} \label{remark: 1} {\rm We emphasise that when $\gamma$ is specifically a horizontal hyperplane, then the stronger estimate
\begin{equation*}
    \int_\gamma |Eg|^2 = \int |g|^2
\end{equation*}
directly follows by Plancherel's theorem. Indeed, for every $(x,t)\in\mathbb{R}^{n-1}\times\mathbb{R}$,
\begin{equation*}
    Eg(x,t)=\int e^{2\pi i \langle x,\omega\rangle}e^{2\pi i t h(\omega)}g(\omega)d\omega=\widehat{g_t}(x),
\end{equation*}
where $g_t:=e^{2\pi i t h(\cdot)}g$ and $\widehat{\cdot}$ denotes the standard Fourier transform on $\mathbb{R}^{n-1}$. Therefore,
\begin{equation*}
    \int |Eg(\cdot, t)|^2=\int |\widehat{g_t}|^2=\int |g_t|^2=\int |g|^2
\end{equation*}
for all $t\in \mathbb{R}$. (Note that this directly yields \eqref{eq: easy 2}.) After an appropriate change of variables, a similar argument resolves the Mizohata--Takeuchi conjecture when the weight is the indicator function of the $1$-neighbourhood of any hyperplane (independently of orientation), and subsequently when the weight is a sum of indicator functions of such $1$-neighbourhoods. See \cite[Corollary 3]{BNS22} for a stronger estimate (a certain identity) in this specific scenario.}
\end{remark}

\hfill $\blacksquare$

Lemma \ref{lemma: single flake} easily implies the Mizohata--Takeuchi conjecture for superpositions of appropriately large flakes, and in fact an estimate stronger than Stein's conjecture \eqref{eq: Stein's conjecture}. 

\begin{corollary}\textbf{\emph{(MT holds for $R^{1/2}$-flakes)}}\label{corollary: flakes} The inequality
\begin{equation*}
    \int_{B_R}|Eg|^2w\lessapprox \|Xw\|_\infty \int |g|^2
\end{equation*}
holds for every $g\in L^2(B^{n-1})$ and any weight $w:\mathbb{R}^n\rightarrow [0,+\infty)$ of the form $\sum_{\gamma\in\mathcal{F}}c_\gamma\chi_\gamma$, where $\mathcal{F}$ is a family of $R^{1/2}$-flakes. In fact, the stronger estimate
\begin{equation*}
    \int_{B_R}|Eg|^2w\lessapprox \sum_{T\in\mathbb{T}}\; \sup_{\ell\subset T}Xw(\ell)\int |g_T|^2
\end{equation*}
holds, where $\{g_T\}_{T\in\mathbb{T}}$ is the wave packet decomposition of $g$ at scale $R$.
\end{corollary}

\begin{proof}
Fix $g:L^2(B^{n-1})$ and $\gamma\in\mathcal{F}$, and denote by $\mathbb{T}_\gamma$ the set of tubes in $\mathbb{T}$ that intersect $\gamma$. For all $x\in\gamma$,
\begin{equation*}
    Eg(x)=Eg_\gamma(x)+{\rm RapDec}_\epsilon(R)\int |g|^2,
\end{equation*}
where $g_\gamma:=\sum_{T\in\mathbb{T}_\gamma}g_T$. Hence, by Lemma \ref{lemma: single flake},
\begin{equation*}
    \int_\gamma |Eg|^2\lessapprox \int |g_\gamma|^2
    \sim \sum_{T\in\mathbb{T}_\gamma}\int |g_T|^2
\end{equation*}
up to an error of ${\rm RapDec}_\epsilon(R)\int |g|^2$. Adding up over all $\gamma\in\mathcal{F}$, we obtain
\begin{eqnarray*}
\begin{aligned}
    \int |Eg|^2w & \lessapprox \sum_{\gamma\in\mathcal{F}}c_\gamma \sum_{T\in\mathbb{T}_\gamma}\int |g_T|^2\\
    &=\sum_{T\in\mathbb{T}}\left(\sum_{\gamma\in\mathcal{F}:\;\gamma\cap T\neq\varnothing}c_\gamma\right) \int |g_T|^2\\
    &\lessapprox \sum_{T\in\mathbb{T}}\; \sup_{\ell\subset T}Xw(\ell) \int |g_T|^2
\end{aligned}
\end{eqnarray*}
up to an error of ${\rm RapDec}_\epsilon(R)\int |g|^2$ (where the final $\approx 1$-loss is due to the fact that the tubes in $\mathbb{T}$ have width $R^{1/2+\delta}$, rather than $R^{1/2}$). The last quantity is at most $\|Xw\|_\infty\int|g|^2$. 

\end{proof}

\begin{remark} {\rm The idea behind the proof of Corollary \ref{corollary: flakes} also appeared in \cite{Gu22}, where the same result was presented in the special case where the flakes are horizontal slabs. Moreover, it was there pointed out that the statement of the corollary also implies \eqref{eq: easy 1}, i.e. that the Mizohata--Takeuchi conjecture holds with loss $\lessapprox R^{1/2}$ in $\mathbb{R}^2$, by replacing each point in ${\rm supp}\;w$ by a horizontal $R^{1/2}$-slab (a process which enlarges the maximal line occupancy of $w$ by $\lesssim R^{1/2}$). Perhaps an easier way to derive \eqref{eq: easy 1} is to observe that, by Proposition \ref{prop: uncertainty principle}, the Mizohata--Takeuchi conjecture holds with $\approx 1$-loss for each function $g_\theta$ supported in an $R^{1/2}$-cap $\theta$; so \eqref{eq: easy 1} follows by the Cauchy--Schwarz inequality, as $B^1$ consists of $\sim R^{1/2}$ such caps.}
\end{remark}

\section{Mizohata--Takeuchi with $R^{\frac{n-1}{n+1}}$-loss: Theorem \ref{theorem: all weights}}\label{section: main proof}

Theorem \ref{theorem: all weights} immediately follows from the stronger Theorem \ref{theorem: all weights refined} below, which takes into account the directions in which the waves propagate. In particular, fix $n\geq 2$. For $g\in L^2(B^{n-1})$ and $\mathbb{T}\subset \mathbb{T}_\epsilon(B_R)$, define
\begin{equation*}
    g_{\mathbb{T}}:=\sum_{T\in\mathbb{T}}g_T,
\end{equation*}
where $\{g_T\}_{T\in\mathbb{T}_\epsilon(R)}$ is the wave-packet decomposition of $g$ adapted to $B_R$ (at scale $R$). 
 
\begin{theorem}\label{theorem: all weights refined} For every $\epsilon>0$, there exists a positive constant $C_\epsilon$, which depends only on $\Sigma$ and $\epsilon$, such that
\begin{eqnarray}\label{eq: all weights refined}
\begin{aligned}
        \int_{B_R} |Eg_\mathbb{T}|^2w \leq &C_\epsilon R^\epsilon \left( \sum_{T \in \mathbb{T}} \Big[\sum_{B\in\mathcal{B}:\;B\cap T\neq\emptyset}w^{\frac{n+1}{2}}(B) \Big]\|g_T\|_2^2\right)^{\frac{2}{n+1}} \|g_\mathbb{T}\|_2^{\frac{2(n-1)}{n+1}} \\
        + &{\rm RapDec}_\epsilon(R)\|w\|_\infty \int |g_{\mathbb{T}}|^2
\end{aligned}
\end{eqnarray}
for all $R\geq 1$, $g\in L^2(\Sigma)$, $\mathbb{T}\subset \mathbb{T}_\epsilon(R)$ and weights $w:\mathbb{R}^n\rightarrow [0,+\infty)$ on $\mathbb{R}^n$, and for every family $\mathcal{B}$ of boundedly overlapping $R^{1/2}$-balls. 
\end{theorem}

As an immediate consequence of this we have:
\begin{corollary}\label{cor: all weights refined} For every $\epsilon>0$, there exists a positive constant $C_\epsilon$, which depends only on $\Sigma$ and $\epsilon$, such that
\begin{eqnarray}\label{eq: all weights refined_1}
\begin{aligned}       
\int_{B_R} |Eg_\mathbb{T}|^2w \leq & C_\epsilon R^\epsilon \left( \int  |g_\mathbb{T}(s)|^2   \sup_{T \parallel N(s), \, T \in \mathbb{T}} w^{\frac{n+1}{2}}(2T) ds \right)^{\frac{2}{n+1}} \|g_\mathbb{T}\|_2^{\frac{2(n-1)}{n+1}} \\
\leq &C_\epsilon R^\epsilon \; \sup_{T\in{\mathbb{T}}} \left(\int_{2T} w^{\frac{n+1}{2}}\right)^{\frac{2}{n+1}}\|g_\mathbb{T}\|_2^2\\
\end{aligned} 
\end{eqnarray}   
up to a ${\rm RapDec}_\epsilon(R)\|w\|_\infty \int |g_{\mathbb{T}}|^2$ error term, for all $R\geq 1$, $g\in L^2(\Sigma)$, $\mathbb{T}\subset \mathbb{T}_\epsilon(R)$ and weights $w:\mathbb{R}^n\rightarrow [0,+\infty)$ on $\mathbb{R}^n$. 
\end{corollary}

\begin{remark}\label{remark: epsilon}
{\rm We need the error term ${\rm RapDec}_\epsilon(R)\|w\|_\infty\int |g_\mathbb{T}|^2$ in these results because $w$ may be large at some points of ${\rm supp} \; Eg_\mathbb{T}$ which are outside $\bigcup_{T \in \mathbb{T}} T$. Theorem~\ref{theorem: all weights refined} manifestly implies Theorem~\ref{theorem: all weights} directly, since the error term is easily absorbed into the right-hand side of the first inequality of Theorem~\ref{theorem: all weights}. It is not possible to take $\epsilon=0$ in either Theorem~\ref{theorem: all weights refined} or in
inequality \eqref{eq: mainineq} of Theorem~\ref{theorem: all weights}. For the case of Theorem~\ref{theorem: all weights refined}, this is
because of the example (see \cite[p.104]{V81}, \cite{R86}, \cite{B93} or \cite[pp.125--126]{V97}) demonstrating the necessity of a logarithmic term in the
discrete $l^2 - L^6$ restriction theorem for the paraboloid. For the
argument linking the two phenomena see \cite[pp.355--358]{BD15}. As we
observe below, Theorem~\ref{theorem: all weights refined} is essentially a reformulation of the
refined decoupling theorem \cite{GIOW20}. For the case of Theorem~\ref{theorem: all weights}, one
may observe directly that with $g$ having all wave packet coefficients
equal, and $w:=|Eg|^{4/(n-1)}$, then $\{w^{(n+1)/2}(T)\}_T$ is uniformly
distributed across the wave packets $T$, and thus the passage from
Theorem~\ref{theorem: all weights refined} to \eqref{eq: mainineq} is tight. (This was noted in discussions between Po
Lam Yung, Zane Li and the first author.) Theorem~\ref{theorem: all weights refined} is furthermore
closely related to the improved decoupling theorem of \cite{GMW20}. More
precisely, if one takes the natural weight $w = |Eg_T|^{4/(n-1)}$ in
Theorem~\ref{theorem: all weights refined}, one obtains an inequality slightly stronger than the one
considered in \cite[Theorem 1.2]{GMW20}, but with $R^\epsilon$ loss rather
than the logarithmic loss obtained there when $n=2$. Notice the Stein-like nature of the middle term appearing in \eqref{eq: all weights refined_1}.

}
\end{remark}

Theorem~\ref{theorem: all weights refined} is  actually a reformulation of the following refined Stein--Tomas or decoupling estimate.  Theorem~\ref{theorem: refined Strichartz} was also discovered independently by Xiumin Du and Ruixiang Zhang (personal
communication).

\begin{theorem}\emph{(Refined decoupling \cite{GIOW20})}\label{theorem: refined Strichartz}
Let $\epsilon>0$, $g\in L^2(B^{n-1})$, and let $\mathbb{T}$ be a subset of $\mathbb{T}_\epsilon(B_R)$ with the property that $\|g_T\|_2$ is roughly constant over all $T\in\mathbb{T}$. For each $k\in\mathbb{N}$, denote by $U_k$ an essentially disjoint union of $R^{1/2}$-balls in $B_R$ each intersecting $\sim k$ tubes in $\mathbb{T}$. Then the function
\begin{equation*}
    g_{\mathbb{T}}=\sum_{T\in\mathbb{T}}g_T
\end{equation*}
satisfies
\begin{equation}
\label{eq: improved ST}
\begin{aligned}
\|Eg_{\mathbb{T}}\|_{L^{\frac{2(n+1)}{n-1}}(U_k)} &\leq C_\epsilon R^\epsilon \left(\frac{k}{\#\mathbb{T}}\right)^{\frac{1}{n+1}} \left(\sum_{T\in\mathbb{T}}\|Eg_T\|_{L^{\frac{2(n+1)}{n-1}}}^2\right)^{1/2} \\
&\sim C_\epsilon R^\epsilon \left(\frac{k}{\#\mathbb{T}}\right)^{\frac{1}{n+1}} \left(\sum_{T\in\mathbb{T}}\|g_T\|_{2}^2\right)^{1/2} \sim C_\epsilon R^\epsilon \left(\frac{k}{\#\mathbb{T}}\right)^{\frac{1}{n+1}} \|g_{\mathbb{T}}\|_{2}. 
\end{aligned}
\end{equation}
\end{theorem}

Since $k\leq \#\mathbb{T}$, estimate \eqref{eq: improved ST} provides an improvement on the classical Stein--Tomas inequality 
\begin{equation*}
    \|Eg_{\mathbb{T}}\|_{L^{\frac{2(n+1)}{{n-1}}}(\mathbb{R}^n)}\lesssim \|g_{\mathbb{T}}\|_2
\end{equation*}
on the `$k$-rich' sets $U_k$ in $B_R$, according to their level $k$ of richness. 

If we assume Theorem \ref{theorem: all weights refined}, we can immediately deduce Theorem~\ref{theorem: refined Strichartz} by testing on a weight $w \in L^{\frac{n+1}{2}}(U_k)$. Indeed, under the hypotheses of \ref{theorem: refined Strichartz}, we apply Theorem \ref{theorem: all weights refined} and we have
\begin{eqnarray*}
\begin{aligned}
        \int_{B_R} |Eg_\mathbb{T}|^2w \leq &C_\epsilon R^\epsilon \left( \sum_{T \in \mathbb{T}} \Big[\sum_{B\in\mathcal{B}:\;B\cap T\neq\emptyset}w^{\frac{n+1}{2}}(B) \Big]\|g_T\|_2^2\right)^{\frac{2}{n+1}} \|g_\mathbb{T}\|_2^{\frac{2(n-1)}{n+1}} \\
        + &{\rm RapDec}_\epsilon(R)\|w\|_\infty \int |g_{\mathbb{T}}|^2
\end{aligned}
\end{eqnarray*}
and, suppressing the error term (as we may) and letting $\lambda = \|g_\mathbb{T}\|_2^2/\# \mathbb{T}$ denote the common value of $\|g_T\|_2^2$, the right hand side here equals 
$$ C_\epsilon R^\epsilon \lambda^{\frac{2}{n+1}} \left( \sum_{T \in \mathbb{T}} \sum_{B\in\mathcal{B}:\;B\cap T\neq\emptyset}w^{\frac{n+1}{2}}(B) \right)^{\frac{2}{n+1}} \|g_\mathbb{T}\|_2^{\frac{2(n-1)}{n+1}}$$ 
$$ \sim C_\epsilon R^\epsilon(\lambda k)^{\frac{2}{n+1}}\left( \sum_{B\in\mathcal{B}} w^{\frac{n+1}{2}}(B)\right)^{\frac{2}{n+1}} \|g_\mathbb{T}\|_2^{\frac{2(n-1)}{n+1}}$$ 
$$ \sim C_\epsilon R^\epsilon\left(\frac{k}{\#\mathbb{T}}\right)^{\frac{2}{n+1}}\|w\|_{\frac{n+1}{2}} \|g_\mathbb{T}\|_2^2,$$
as needed to verify Theorem~\ref{theorem: refined Strichartz}.

Likewise, Theorem \ref{theorem: all weights refined} will in turn follow from \eqref{eq: improved ST}, as the following simple argument shows.

\begin{proof}[Proof of Theorem \ref{theorem: all weights refined}] Let $\epsilon>0$, fix $g\in L^2(B^{n-1})$, $w:B_R\rightarrow [0,+\infty)$ and $\mathbb{T}\subset\mathbb{T}_\epsilon(B_R)$.

In order to prove \eqref{eq: all weights refined}, we may assume that: 
\begin{enumerate}[(a)]
    \item $w$ is supported in $\bigcup_{T\in\mathbb{T}}T$.
    \item $\|g_T\|_2 \sim 1$ for all $T \in \mathbb{T}$.
\end{enumerate}

Indeed, assumption (a) is possible because, by \eqref{eq:wp3}, the part of the weight supported outside $\bigcup_{T\in\mathbb{T}}T$ contributes at most ${\rm RapDec}_\epsilon(R)\|w\|_\infty\int |g_\mathbb{T}|^2$ to $\int_{B_R}|Eg_{\mathbb{T}}|^2 w$. For (b), observe that, in terms of our goal, it is trivial to control the contributions of the wave packets $g_T$ with $\|g_T\|_2<R^{-100n}\|g\|_2$. So, by dyadic pigeonholing, it suffices to prove \eqref{eq: all weights refined} under the additional assumption that the $g_T$ have roughly the same $L^2$ norms over all $T\in\mathbb{T}$. By scaling we may assume this common value is $1$.

We now fix a family $\mathcal{B}$ of boundedly overlapping $R^{1/2}$-balls covering $B_R$. By the above it suffices to prove that
\begin{equation}\label{eq: special case}
    \int|Eg_\mathbb{T}|^2 w\lessapprox \left(\frac{1}{\#\mathbb{T}}\sum_{T\in\mathbb{T}}\;\;\sum_{B\in\mathcal{B}:\;B\cap T\neq \emptyset} w^{\frac{n+1}{2}}(B)\right)^{\frac{2}{n+1}}\int |g_\mathbb{T}|^2
\end{equation}
under assumptions (a) and (b).

Let $U_k$ be the union of the balls in this family which meet $\sim k$ members of $\mathbb{T}$.

Importantly, (a) ensures that there exists some dyadic $k\in\mathbb{N}$ for which
\begin{equation*}
    \int_{B_R}|Eg_\mathbb{T}|^2 w \approx \int_{U_k}|Eg_\mathbb{T}|^2 w,
\end{equation*}
So by H\"older's inequality and \eqref{eq: improved ST} we obtain
\begin{eqnarray*}
\begin{aligned}
    \int_{B_R}|Eg_\mathbb{T}|^2 w &\lessapprox \left(\int_{U_k}|Eg_\mathbb{T}|^\frac{2(n+1)}{n-1}\right)^{\frac{n-1}{n+1}}\; (w^{\frac{n+1}{2}}(U_k))^{\frac{2}{n+1}}\\
    &\leq C_\epsilon R^\epsilon \left(\frac{k}{\#\mathbb{T}}\;w^{\frac{n+1}{2}}(U_k)\right)^{\frac{2}{n+1}}\int |g_\mathbb{T}|^2\\
    & \sim C_\epsilon R^\epsilon \left(k\;w^{\frac{n+1}{2}}(U_k)\right)^{\frac{2}{n+1}} (\# \mathbb{T})^{\frac{n-1}{n+1}}.
\end{aligned}
\end{eqnarray*}
We conclude with a simple counting argument. Indeed, let $\mathcal{B}_k$ be the set of $R^{1/2}$-balls comprising $U_k$. Then, 
\begin{eqnarray*}
\begin{aligned}
    k \; w^{\frac{n+1}{2}}(U_k)&\sim\sum_{B\in\mathcal{B}_k}w^{\frac{n+1}{2}}(B) \; k\\
    & \sim \sum_{B\in\mathcal{B}_k}\;\;\sum_{T\in \mathbb{T}:\; T\cap B\neq \emptyset}w^{\frac{n+1}{2}}(B)\\
    &=\sum_{T\in\mathbb{T}}\;\;\sum_{B\in\mathcal{B}_k:\;B\cap T\neq \emptyset} w^{\frac{n+1}{2}}(B),
  \end{aligned}
\end{eqnarray*}
establishing \eqref{eq: special case} and thus \eqref{eq: all weights refined}.
\end{proof}


\section{Improved Mizohata--Takeuchi estimates for small caps}\label{section: small caps}

In this section we prove Lemma \ref{lemma: small caps}, which will be key to the proofs of Theorems \ref{theorem: hor slabs refined} and \ref{theorem: all slabs refined}.
It is a Mizohata--Takeuchi-type estimate which holds for functions supported in small caps, and it represents an improvement over what we can obtain under no support hypothesis. 

Towards proving the lemma, we may assume as in Section~\ref{section: preliminaries} that all normals to $\Sigma$ have angle at most $1/100$ from the vertical direction, and that the projection of $\Sigma$ on the hyperplane $\mathbb{R}^{n-1}\times \{0\}$ is contained in the unit ball $B^{n-1}$ centred at 0. It thus suffices to establish the analogous statement (Lemma~\ref{lemma: small caps'} below) with $Eg_\tau$ in place of $\widehat{g_\tau d\sigma}$, where $E$ is the extension operator associated to $\Sigma$ and $g_\tau\in L^2(B^{n-1})$ is a function supported in a $\rho^{-1/2}$-cap $\tau$ in $B^{n-1}$. 

To simplify notation, for $E\subset B^{n-1}$ (rather than $E\subset \Sigma$), and any line $\ell$ (or tube $T$ in $B_R$), we write $\ell\perp E$ if $\ell\perp \Sigma(E)$ (similarly, we write $T\perp E$ if $T\perp \Sigma(E)$). We also define
\begin{equation*}
    A_{\rho,R,E}(w):=A_{\rho,R,\Sigma(E)}(w).
\end{equation*}  

\begin{lemma}\label{lemma: small caps'} For every $\epsilon>0$, there exists $C_\epsilon>0$ such that
for all weights $w:\mathbb{R}^n\rightarrow [0,+\infty)$, whenever $1\leq\rho\leq R$, $\tau$ is a $\rho^{-1/2}$-cap in $B^{n-1}$ and $g_\tau\in L^2(B^{n-1})$ is supported in $\tau$, we have
\begin{equation*}
    \int_{B_R}|Eg_\tau|^2 w \leq C_\epsilon R^\epsilon\;  A_{\rho,R, \;{\rm supp}\;g_\tau}(w) \int |g_\tau|^2,
\end{equation*}
and therefore also
\begin{equation*}
    \int_{B_R}|Eg_\tau|^2 w \leq C_\epsilon R^\epsilon\; \left(\frac{R}{\rho}\right)^{\frac{n-1}{n+1}}\sup_{\ell\perp {\rm supp}\;g_\tau} Xw(\ell)  \int |g_\tau|^2.
\end{equation*}
\end{lemma}

Notice that the tubes and lines featuring here have directions perpendicular to the support of $g_\tau$.

\begin{proof} Let $\epsilon>0$ and $R\geq 1$. For $\rho\lesssim R^\epsilon$, the conclusion of the lemma follows directly from Theorem \ref{theorem: all weights}. We therefore consider $\rho\gtrsim R^\epsilon$. 

In order to prove the lemma for arbitrary weights, it suffices by dyadic pigeonholing to prove it for weights that are indicator functions. Indeed, first observe that we may assume that $w(x)\geq R^{-2n}\|w\|_\infty$ for all $x\in {\rm supp}\;w$. Therefore, after a dyadic pigeonholing causing losses of $\sim \log R$, we may assume that $w(x)\sim q$ for some fixed $q>0$ over all $x\in {\rm supp}\;w$; and hence that $w$ is an indicator function, due to the scaling properties of our desired estimate.

So, let $w$ be an indicator function of a non-empty union of unit balls. Fix a $\rho^{-1/2}$-cap $\tau$, and let $g$ be a function supported in $\tau$. Let $\mathbb{T}$ be a family of boundedly overlapping parallel $\rho^{1/2}$-tubes that cover ${\rm supp}\;w$, and point in some direction $N$ normal to ${\rm supp}\;g$; observe that $\mathbb{T}\subset\mathbb{T}_\rho$. At a cost of a $\log R$-loss, it may be further assumed that
\begin{equation*}
    \frac{w(S_\rho)}{|S_\rho|}\sim \lambda\text{ for all }S_\rho\in\mathbb{T}
\end{equation*}
for some $\lambda\leq 1$, hence
\begin{eqnarray*}
\begin{aligned}
    A_{\rho,R,\;{\rm supp}\;g}(w)&=\sup_{T_R\in\mathbb{T}_R:\;T_R\perp {\rm supp}\;g}\left(\sum_{S_\rho\subset T_R}\left(\frac{w(S_\rho)}{|S_\rho|}\right)^{\frac{n+1}{2}}|S_\rho|\right)^{\frac{2}{n+1}}\\
    &\sim  \lambda \rho \sup_{T_R\in \mathbb{T}_R:\; T_R\perp {\rm supp}\;g} \#\{S_\rho \in \mathbb{T} \, : \, S_\rho \cap T_R\neq\emptyset\}^{\frac{2}{n+1}}.
\end{aligned}
\end{eqnarray*}

It therefore suffices to prove that 
\begin{equation*}
    \int |Eg|^2 w 
   \leq C_\epsilon R^\epsilon \lambda \rho \sup_{T_R\in \mathbb{T}_R:\; T_R\perp {\rm supp}\;g} \#\{S_\rho \in \mathbb{T} \, : \, S_\rho \subset T_R\}^{\frac{2}{n+1}} \int |g|^2.
\end{equation*}
Proposition \ref{prop: uncertainty principle} ensures that, roughly speaking, $|Eg|$ is constant on each $S_\rho\in\mathbb{T}$. In particular, let $\mathbb{T}_N$ be a set of boundedly overlapping tubes in direction $N$, of width $\rho^{1/2+\delta}$ and length $\rho$, that cover $B_R$. For each $S_\rho\in\mathbb{T}$, fix $\widetilde{S}_\rho\in\mathbb{T}_N$ that intersects $S_\rho$. By Proposition \ref{prop: uncertainty principle},
\begin{align*}
    \int_{S_\rho} |Eg|^2 w&\lesssim \frac{w(S_\rho)}{|S_\rho|}\int_{2\widetilde{S}_\rho}|Eg|^2 +{\rm RapDec}_\epsilon (R) \int |g|^2\\
    &\sim \lambda \int_{2\widetilde{S}_\rho}|Eg|^2 +{\rm RapDec}_\epsilon (R) \int |g|^2.
\end{align*}
By adding over all $S_\rho\in\mathbb{T}$, we obtain
\begin{equation*}
    \int |Eg|^2 w \lesssim \lambda \int |Eg|^2 \widetilde{w} +{\rm RapDec}_\epsilon (R) \int |g|^2,
\end{equation*}
where
\begin{equation*}
    \widetilde{w}:=\sum_{S_\rho\in\mathbb{T}}\chi_{2\widetilde{S}_\rho}.
\end{equation*}
Now by Theorem~\ref{theorem: all weights} we have
\begin{equation*}\int|Eg|^2 \tilde{w} \lessapprox \sup_{T_R \in \mathbb{T}_R: \; T_R \perp {\rm supp}\; g} \tilde{w}(T_R)^\frac{2}{n+1} \int |g|^2,
\end{equation*}
and for $T_R \in \mathbb{T}_R$ with $T_R \perp {\rm supp}\; g$ we have
\begin{equation*}\tilde{w}(T_R) \lesssim \rho^{\frac{n+1}{2}} \# \{S_\rho \in \mathbb{T} \, : \, 2S_\rho \cap T_R\neq\varnothing\}.
\end{equation*}
Therefore,
\begin{equation*}
\int|Eg|^2 w \lessapprox \lambda \left( \rho^{\frac{n+1}{2}} \sup_{T_R \in \mathbb{T}_R: \; T_R \perp {\rm supp}\; g} \# \{S_\rho \in \mathbb{T} \, : \, S_\rho \subset T_R\}\right)^{\frac{2}{n+1}}
\int |g|^2,
\end{equation*}
as required.
\end{proof}


\section{Weights constant on slabs: Theorems \ref{theorem: hor slabs} and \ref{theorem: all slabs}}\label{section: slabs}

In this section we will use the favourable estimates for functions $g_\tau$ supported in small caps which were established in Section~\ref{section: small caps} to obtain Mizohata--Takeuchi estimates which improve on Theorem~\ref{theorem: all weights} for general functions $g$ and weights possessing a certain measure of local constancy. In particular, recall from \eqref{eq:MTsmallcaps} that if a function $g_\tau$ is supported in a $\rho^{-1/2}$-cap $\tau$, then the Mizohata--Takeuchi conjecture holds for $g_\tau$ with an improved $(R/\rho)^{(n-1)(n+1)}$-loss. Therefore, for any fixed $g\in L^2(B^{n-1})$ and $w:\mathbb{R}^n\rightarrow [0,+\infty)$, a decoupling inequality of the form
\begin{equation*}
    \int_{B_R}|Eg|^2w \lessapprox \sum_{\tau}\int_{B_R}|Eg_\tau|^2w
\end{equation*}
for a boundedly overlapping collection of $\rho^{-1/2}$-caps $\tau$ (where $g=\sum_{\tau}g_\tau$ and ${\rm supp}\;g_{\tau}\subset\tau$) would directly imply that Mizohata--Takeuchi holds for $g$ with the inherited loss $(R/\rho)^{(n-1)(n+1)}$. The smaller the caps we manage to decouple into, the smaller the loss.

In general, it is not possible to decouple into small caps. However, we can indeed decouple into $\rho^{-1/2}$-caps when $w$ is a weight of the form $\sum_{s\in\mathcal{S}}c_s\chi_s$, where $\mathcal{S}$ is a set of disjoint $\rho^{1/2}$-slabs that are $\nu$-parallel to $\Sigma$; more precisely, we show that \eqref{eq: decoupling} below holds. This yields Mizohata--Takeuchi for such weights with an $(R/\rho)^{(n-1)(n+1)}$-loss. If the slabs in $\mathcal{S}$ are allowed to point in any direction, then we can decouple into larger $\rho^{-1/4}$-caps \eqref{eq: decoupling'}, inheriting Mizohata--Takeuchi with an $(R/\rho^{1/2})^{(n-1)(n+1)}$-loss.

These results are given in Theorems \ref{theorem: hor slabs refined} and \ref{theorem: all slabs refined} below, which are more precise versions of Theorems \ref{theorem: hor slabs} and \ref{theorem: all slabs}, respectively. As per the above discussion, the new ingredients here are the decoupling inequalities \eqref{eq: decoupling} and \eqref{eq: decoupling'} which follow. Note that, as in Section~\ref{section: small caps}, we will be working with the extension operator $E$ associated to $\Sigma$ (rather than with $\widehat{\cdot\;d\sigma}$). When $E\subset B^{n-1}$, we will be using the simpler the notation $A_{\rho,R, E}(w)$ in place of $A_{\rho,R,\Sigma(E)}(w)$, and $\ell\perp E$ (or $T\perp E$) to mean $\ell\perp \Sigma(E)$ (similarly, $T\perp \Sigma(E)$) for any line $\ell$ and tube $T$ in $\mathbb{R}^n$.

\begin{theorem}\textbf{\emph{(Roughly horizontal slabs)}}\label{theorem: hor slabs refined} Fix $\nu>0$ and $\epsilon>0$. For $1\leq\rho\leq R$, let $w:\mathbb{R}^n\rightarrow [0,+\infty)$ be a weight of the form $\sum_{s\in\mathcal{S}}c_s\chi_s$, where $\mathcal{S}$ is a set of disjoint $\rho^{1/2}$-slabs $\nu$-parallel to $\Sigma$, and let 
$w^\star:=\sum_{s\in\mathcal{S}}c_s\chi_{3s}$. For $g\in L^2(B^{n-1})$, write
\begin{equation*}
    g=\sum_{\tau\in\mathfrak{T}}g_\tau,\quad{{\rm supp}\;g_\tau\subset\tau},
\end{equation*}
where $\mathfrak{T}$ is a family of boundedly overlapping $\rho^{-1/2}$-caps $\tau$ in $B^{n-1}$. Then the decoupling inequality
\begin{equation}\label{eq: decoupling}
    \int_{B_R}|Eg|^2w \lessapprox_\nu \sum_{\tau\in\mathfrak{T}}\int_{B_R}|Eg_\tau|^2w^\star + {\rm RapDec}_\epsilon(R)\int |g|^2
\end{equation}
holds. Consequently we have 
\begin{eqnarray}
\begin{aligned}\label{eq: final}
    \int_{B_R} |Eg|^2w&\leq C_{\epsilon,\nu}R^\epsilon  \sum_{\tau\in\mathfrak{T}}A_{\rho, R, \; {\rm supp}\; g_\tau}(w)\int |g_\tau|^2\\
    &\lessapprox_\nu \left(\frac{R}{\rho}\right)^{\frac{n-1}{n+1}} \sum_{\tau\in\mathfrak{T}} \; \sup_{\ell\perp {\rm supp}\; g_\tau}Xw(\ell)\int |g_\tau|^2.
\end{aligned}
\end{eqnarray}
\end{theorem}
Note that an immediate consequence of \eqref{eq: final} is 
\begin{eqnarray*}
\begin{aligned}
 \int_{B_R} |Eg|^2w&\leq C_{\epsilon,\nu}R^\epsilon  A_{\rho, R, \; {\rm supp}\; g}(w)\int |g|^2\\
    &\lessapprox_\nu \left(\frac{R}{\rho}\right)^{\frac{n-1}{n+1}} \sup_{\ell\perp {\rm supp}\; g}Xw(\ell)\int |g|^2.
\end{aligned}
\end{eqnarray*}

\vspace{0.1in}

\begin{theorem}\textbf{\emph{(All slabs)}}\label{theorem: all slabs refined} Fix $\epsilon>0$. For $1\leq \rho \leq R$, let $w:\mathbb{R}^n\rightarrow [0,+\infty)$ be a weight of the form $\sum_{s\in\mathcal{S}}c_s\chi_s$, where $\mathcal{S}$ is a set of disjoint $\rho^{1/2}$-slabs. Let $w^\star:=\sum_{s\in\mathcal{S}}c_s\chi_{3s}$. For $g\in L^2(B^{n-1})$, write
\begin{equation*}
    g=\sum_{\widetilde{\tau}\in\widetilde{\mathfrak{T}}}g_{\widetilde{\tau}},\quad{{\rm supp}\;g_{\widetilde{\tau}}\subset\widetilde{\tau}}
\end{equation*}
where $\widetilde{\mathfrak{T}}$ is a family of finitely overlapping $\rho^{-1/4}$-caps $\widetilde{\tau}$ in $B^{n-1}$. Then the decoupling inequality
\begin{equation}\label{eq: decoupling'}
    \int_{B_R}|Eg|^2w \lessapprox \sum_{\widetilde{\tau}\in\widetilde{\mathfrak{T}}}\int_{B_R}|Eg_{\widetilde{\tau}}|^2w^\star + {\rm RapDec}_\epsilon(R)\int |g|^2
\end{equation}
holds. Consequently we have,
\begin{eqnarray}
\begin{aligned}\label{eq: final'}
    \int_{B_R} |Eg|^2w&\leq C_{\epsilon}R^\epsilon  \sum_{\widetilde{\tau}\in\widetilde{\mathfrak{T}}}A_{\rho^{1/2},R,\; \rm{supp}\,g_{\widetilde{\tau}}}(w)\int |g_{\widetilde{\tau}}|^2\\
    &\lessapprox \left(\frac{R}{\rho^{1/2}}\right)^{\frac{n-1}{n+1}} \sum_{\widetilde{\tau}\in\widetilde{\mathfrak{T}}} \; \sup_{\ell\perp {\rm supp}\; g_{\widetilde{\tau}}}Xw(\ell)\int |g_{\widetilde{\tau}}|^2.
\end{aligned}
\end{eqnarray}

\end{theorem}

Note that an immediate consequence of \eqref{eq: final'} is 
\begin{eqnarray*}
\begin{aligned}
 \int_{B_R} |Eg|^2w&\leq C_{\epsilon,\nu}R^\epsilon  A_{\rho^{1/2}, R, \; {\rm supp}\; g}(w)\int |g|^2\\
    &\lessapprox_\nu \left(\frac{R}{\rho^{1/2}}\right)^{\frac{n-1}{n+1}} \sup_{\ell\perp {\rm supp}\; g}Xw(\ell)\int |g|^2.
\end{aligned}
\end{eqnarray*}

\begin{proof}[Proofs of \eqref{eq: decoupling} and \eqref{eq: decoupling'}.] Fix $\epsilon>0$ and $R\geq 1$. Let $s$ be a $\rho^{1/2}$ slab in $B_R$, and fix $g\in L^2(B^{n-1})$. Let $\mathfrak{T}_1$, $\mathfrak{T}_2$ be collections of finitely overlapping $\rho^{-1/4}$ and $\rho^{-1/2}$-caps, respectively, that cover $B^{n-1}$. For $i=1,2$, write
\begin{equation*}
    g=\sum_{\tau\in \mathfrak{T}_i}g_\tau\qquad{{\rm supp}\,g_\tau\subset \tau}.
\end{equation*}
We will show that
\begin{equation*}
    \int_s |Eg|^2\leq C_\epsilon R^\epsilon \sum_{\tau\in\mathfrak{T}_1}\int_{3s} |Eg_\tau|^2
\end{equation*}
and that, if additionally $s$ is $\nu$-parallel to $\Sigma$ for some $\nu>0$, then
\begin{equation*}
    \int_s |Eg|^2\leq C_{\nu,\epsilon} R^\epsilon \sum_{\tau\in\mathfrak{T}_2}\int_{3s} |Eg_\tau|^2.
\end{equation*}
Note that henceforth we may assume that $\rho\gtrsim_\epsilon R^{\epsilon/n}$ (as otherwise \eqref{eq: decoupling} and \eqref{eq: decoupling'} follow trivially by the Cauchy--Schwarz inequality), and that $\nu\gtrsim_\epsilon R^{-\epsilon}$ (as otherwise $C_{\epsilon,\nu}$ may be chosen to be an appropriately large power of $R$ for \eqref{eq: decoupling} to follow).

For this proof, it will be useful to think of $g$ as truly supported on $\Sigma$. And indeed, due to our assumption that the normals to $\Sigma$ create angles at most $1/100$ with the vertical direction, it suffices instead to prove the above decoupling inequalities for $g\in L^2(\Sigma)$, for $\widehat{gd\sigma}$ in place of $Eg$ and for $\mathfrak{T}_i$ collections of finitely overlapping $\rho^{-1/4}$-caps and $\rho^{-1/2}$-caps, respectively, of $\Sigma$.

Let $\eta:\mathbb{R}^n\rightarrow\mathbb{R}$ be a non-negative, smooth bump function with $\eta(x)=1$ for all $x\in B_1$ and $\eta(x)=0$ for all $x\in B_2$. Denote by $\eta_s$ a smooth bump function adapted to $s$. In particular, if $s_0=[0,\rho^{1/2}]^{n-1}\times [0,1]$, define
\begin{equation*}
    \eta_{s_0}(x):=\eta\left(\frac{x'}{\rho^{1/2}},x_n\right),
\end{equation*}
and let $\eta_s(x):=\eta_{s_0}(Mx)$, where $M$ is a rigid motion mapping $s$ to $s_0$. Let $s^\star$ be a `dual' object to $s$, specifically the tube with centre 0, direction the normal to $s$, length 1 and cross section of radius $\rho^{-1/2+\delta}$. It is easy to see by stationary phase that $\widehat{\eta_s(x)}$ is essentially supported in $s^\star$; more precisely,
\begin{equation*}
    |\widehat{\eta_s}(y)|={\rm RapDec}_\epsilon(\rho)\|\eta_s\|_1={\rm RapDec}_\epsilon(R)\qquad{\text{for all }y\in \mathbb{R}^n\setminus s^\star}.
\end{equation*}
Therefore, for $i=1,2$,
\begin{align*}
    \int_s |\widehat{gd\sigma}|^2\leq \int |\widehat{gd\sigma}|^2\eta_s&=\int \Big|\sum_{\tau\in\mathfrak{T}_i} \widehat{g_\tau d\sigma}\Big|^2\eta_s\\
    &=\int \Big(\sum_{\tau\in\mathfrak{T}_i} \widehat{g_\tau d\sigma}\Big)\Big(\sum_{\tau'\in\mathfrak{T}_i}  \overline{\widehat{g_{\tau'}d\sigma}}\Big)\eta_s\\
    &=\sum_{\tau,\tau' \in \mathfrak{T}_i}\int \left(\widehat{g_\tau d\sigma} \;\overline{\widehat{g_{\tau'}d\sigma}}\right)\eta_s\\
    &=\sum_{\tau,\tau'\in\mathfrak{T}_i}\int (g_\tau d\sigma) \ast (\widetilde{g_{\tau'}d\sigma})\;\widehat{\eta_s},
\end{align*}
where, for every $f:\mathbb{R}^n\rightarrow \mathbb{C}$, $\widetilde{f}$ is defined by $\widetilde{f}(y):=\overline{f(-y)}$.

For every $\tau,\tau'\in\mathfrak{T}_i$, the function $(g_\tau d\sigma) \ast (\widetilde{g_{\tau'}d\sigma})$ is supported in $\tau-\tau'$, and thus its contribution to the above sum is negligible unless $\tau-\tau'$ intersects $s^\star$. More precisely,
\begin{equation*}
    \int (g_\tau d\sigma) \ast (\widetilde{g_{\tau'}d\sigma})\;\widehat{\eta_s}=\int_{\mathbb{R}^n\setminus s^\star} (g_\tau d\sigma) \ast (\widetilde{g_{\tau'}d\sigma})\;\widehat{\eta_s}={\rm RapDec}_\epsilon (R)\|g_\tau\|_2\|g_{\tau'}\|_2
\end{equation*}
whenever $\tau-\tau'\cap s^\star=\varnothing$, whence
\begin{eqnarray}\label{eq: dec}
\begin{aligned}
    \int_s |\widehat{gd\sigma}|^2&=\sum_{\tau,\tau'\in\mathfrak{T}_i:\; (\tau-\tau')\cap s^\star\neq\varnothing}\int (g_\tau d\sigma) \ast (\widetilde{g_{\tau'}d\sigma})\;\widehat{\eta_s}+ {\rm RapDec}_\epsilon (R) \int |g|^2\\
    &=\sum_{\tau,\tau'\in\mathfrak{T}_i:\; (\tau-\tau')\cap s^\star\neq\varnothing}\int \left(\widehat{g_\tau d\sigma} \;\overline{\widehat{g_{\tau'}d\sigma}}\right)\eta_s+ {\rm RapDec}_\epsilon (R) \int |g|^2\\
    &\leq \sum_{\tau,\tau'\in\mathfrak{T}_i:\; (\tau-\tau')\cap s^\star\neq\varnothing} \left( \int_{3s} |\widehat{g_\tau d\sigma}|^2+\int_{3s}|\widehat{g_{\tau'}d\sigma}|^2\right) + {\rm RapDec}_\epsilon (R) \int |g|^2\\
    &\leq N_i\cdot \sum_{\tau\in\mathfrak{T}_i}\int_{3s} |\widehat{g_\tau d\sigma}|^2+ {\rm RapDec}_\epsilon (R) \int |g|^2,
\end{aligned}
\end{eqnarray}
where
\begin{equation*}
    N_i:=\max_{\tau\in\mathfrak{T}_i}\;\#\{\tau'\in\mathfrak{T}_i:(\tau-\tau')\cap s^\star\neq\varnothing\}.
\end{equation*}
Note that for the last inequality in \eqref{eq: dec} we used that $s^\star$ is symmetric around 0.

It now suffices to show that 
\begin{equation}\label{eq: card1}
    N_1\leq C_\epsilon R^\epsilon
\end{equation}
and that, if additionally $s$ is $\nu$-parallel to $\Sigma$ for some $\nu\gtrsim_\epsilon R^\epsilon$, then
\begin{equation}\label{eq: card2}
    N_2\leq C_{\epsilon,\nu} R^\epsilon.
\end{equation}
We first focus on the case $i=1$. Fix $\tau \in \mathfrak{T}_1$, and let $\omega(\tau)$ denote its centre. The family $\mathfrak{T}_1$ consists of $\rho^{-1/4}$-caps, so the $\tau'\in\mathfrak{T}_1$ with $(\tau-\tau')\cap s^\star\neq\varnothing$ cover the set
\begin{equation*}
    A(\tau):=\{\omega\in\Sigma: \; (\tau-\omega)\cap s^\star\neq \varnothing\}.
\end{equation*}
Let $e$ denote the direction of $s^\star$. For every $\omega\in A(\tau)$, there exists $\omega_0\in\tau$ such that $\omega_0-\omega \in s^\star$, which implies that
\begin{equation*}
    |\omega_0-\omega|\lesssim \rho^{-1/4+\delta}\;\;\;\text{ or }\;\;\;{\rm Angle}(\omega_0-\omega, e)\lesssim \rho^{-1/4+\delta},
\end{equation*}
hence
\begin{equation*}
    |\omega-\omega(\tau)|\lesssim \rho^{-1/4+\delta}\;\;\;\text{ or }\;\;\;{\rm Angle}(\omega-\omega(\tau), e)\lesssim \rho^{-1/4+\delta}.
\end{equation*}
It follows that $A(\tau)$ can be covered by two $\sim\rho^{-1/4+\delta}$-caps of $\Sigma$, and thus by $O(\rho^\delta)=O(R^\epsilon)$ $\rho^{-1/4}$-caps of $\Sigma$. This immediately implies \eqref{eq: card1}, which in turn establishes the desired estimate \eqref{eq: decoupling'} when combined with \eqref{eq: dec}.

For the case $i=2$, let $\nu\gtrsim_\epsilon R^\epsilon$. Fix $\tau\in \mathfrak{T}_2$ and denote by $\omega(\tau)$ its centre. Similarly to the previous case, the $\tau'\in\mathfrak{T}_2$ with $(\tau-\tau')\cap s^\star\neq\varnothing$ cover the set
\begin{equation*}
    A(\tau):=\{\omega\in\Sigma: \; (\tau-\omega)\cap s^\star\neq \varnothing\}=\Sigma\cap (\tau-s^\star).
\end{equation*}
Now however the family $\mathfrak{T}_2$ consists of $\rho^{-1/2}$-caps; moreover, $s$ is $\nu$-parallel to $\Sigma$, which implies that all tangents to $\tau$ create angles at least $\nu$ with the (roughly vertical) direction $e$ of $s^\star$. Therefore,
\begin{equation*}
    \tau-s^\star\subset R_{s^\star},
\end{equation*}
for some vertical rectangle $R_{s^\star}$, with vertical side of length $\sim_\nu 1$ (roughly the length of $s^\star$) and all other sides of length $\sim_\nu \rho^{-1/2+\delta}$ (approximately the sum of the width of $s^\star$ and the radius of $\tau$). 

Due to our assumption that all tangents to $\Sigma$ create angle at most $1/100$ with the vertical direction, it follows that $\Sigma \cap R_{s^\star}$ (and consequently $A(\tau)$) is contained in a single $\sim_\nu \rho^{-1/2+\delta}$-cap of $\Sigma$, and can thus be covered by $O(R^\epsilon)$ $\rho^{-1/2}$-caps in $\mathfrak{T}_2$. This implies the desired estimate \eqref{eq: card2} and hence completes the proof of \eqref{eq: decoupling}.
\end{proof}

\begin{proof}[Proof of Theorem \ref{theorem: hor slabs refined}.] Let $\nu$, $\epsilon$, $R$, $\rho$, $w$ and $g$ be as in the statement of the theorem. Now that \eqref{eq: decoupling} has been established, it suffices to prove the first assertion in \eqref{eq: final}.

To that end, observe that $w^\star$ is the sum of $3^{n-1}$ weights: the weight $w_0:=w$ (supported in $B_R$), and weights $w_j$ of the form $w(\cdot - t_j)$ (for appropriate $t_j\in\mathbb{R}^{n-1}\times \{0\}$, with $|t_j| \leq R$, for $j=1,2\dots$). It thus suffices to show that
\begin{equation*}
    \int |Eg|^2w_j\leq C_{\epsilon, \nu}R^\epsilon \sum_{\tau\in\mathfrak{T}}A_{\rho, R, \,{\rm supp}\,g_\tau}(w)\int |g_\tau|^2
\end{equation*}
for all $j=1,2 \dots$. For $j=0$ the inequality follows by Lemma \ref{lemma: small caps'}. For $j=1,2 \dots$,
\begin{equation*}
    Eg=Eg_j(\cdot - t_j)\text{, where }g_j:=e^{2\pi i \langle t_j,\Sigma(\cdot)\rangle}g.
\end{equation*}
Observe that, denoting $g_{j,\tau}:=e^{2\pi i \langle t_j,\Sigma(\cdot)\rangle}g_\tau$, we can write
\begin{equation*}
    g_j=\sum_{\tau\in\mathfrak{T}}g_{j,\tau},\quad{{\rm supp}\;g_{j,\tau}={\rm supp}\;g_{\tau}\subset \tau}.
\end{equation*}
Therefore, by Lemma \ref{lemma: small caps'},
\begin{align*}
    \int |Eg|^2w_j&=\int |Eg_j(\cdot-t_j)|^2w(\cdot-t_j)\\
        &=\int |Eg_j|^2w\\
        &\leq C_{\epsilon,\nu} R^\epsilon \sum_{\tau\in\mathfrak{T}}A_{\rho, R, \,{\rm supp}\,g_{j, \tau}}(w)\int |g_{j,\tau}|^2\\
        &=C_{\epsilon,\nu} R^\epsilon \sum_{\tau\in\mathfrak{T}}A_{\rho, R, \,{\rm supp}\,g_\tau}(w)\int |g_{\tau}|^2,
\end{align*}
completing the proof.
\end{proof}

\begin{proof}[Proof of Theorem \ref{theorem: all slabs refined}.] The proof follows the same steps as that of Theorem \ref{theorem: hor slabs refined}, but with the family $\mathfrak{T}$ replaced by $\widetilde{\mathfrak{T}}$.
\end{proof}


\section{Guth's argument: the $R^{\frac{n-1}{n+1}}$ barrier}\label{section: Guth's example}

In his recent talk \cite{Gu22}:
\begin{enumerate}[(a)]
    \item Guth identified two `decoupling axioms' (appropriate local constancy and local $L^2$-orthogonality conditions) that are satisfied by all $Eg$, and are sufficient to ensure that the Bourgain--Demeter decoupling inequality \cite{BD15} holds in $B_R$ for every function $F$ satisfying them.
    \item He then constructed a function $F:B_R\rightarrow\mathbb{C}$ which satisfies the decoupling axioms, but for which the Mizohata--Takeuchi conjecture fails by a factor of $\sim (\log R)^{-3} R^{\frac{n-1}{n+1}}$. Notably, $F_{|B_R}$ is not of the form $Eg_{|B_R}$ for any $g\in L^2(B^{n-1})$.
\end{enumerate}

Notably, Guth's decoupling axioms for all $Eg$ are also sufficient to imply the refined decoupling Theorem \ref{theorem: refined Strichartz} (as a careful review of its proof reveals), and thus its corollary Theorem \ref{theorem: all weights}, which established the conjecture with a loss of $\lessapprox R^{\frac{n-1}{n+1}}$. Therefore, our main result is essentially sharp given the techniques used.

In this section we outline Guth's axiomatic approach and argument demonstrating the existence of a counterexample \cite{Gu22}, and briefly review our result within this context. We emphasise that these results are not ours, and we present them only for self-containment.

Fix $R\geq 1$ and $\epsilon>0$. In this section, for every $g\in L^2(B^{n-1})$ and every cap $\tau$ in $B^{n-1}$, we denote $g_\tau:=g_{|\tau}$. In particular, $g_{B^{n-1}}=g$.

We call a cap $\tau$ in $B^{n-1}$ \textit{admissible} if its diameter $d(\tau)$ is a dyadic number in $[R^{-1/2}, R^{-\epsilon}]\cup \{2\}$. In this analysis, $B^{n-1}$ is the only admissible cap of diameter 2. Denote by $\mathcal{D}_R$ the set of all admissible caps. 

For every $\tau\in \mathcal{D}_R$, let $F_\tau:\mathbb{R}^n\rightarrow\mathbb{C}$ be some function. Note that the caps $\tau$ are simply used for enumeration here, and may be entirely unrelated to properties of $F_\tau$. This is in contrast to, say, functions of the form $Eg_\tau$, which are Fourier-localised close to $\Sigma(\tau)$.

\textbf{Axiomatic decoupling. (Guth \cite{Gu22})} \textit{If the decoupling axioms} (DA1) \textit{and} (DA2) \textit{below hold for the full sequence $(F_\tau)_{\tau\in D_R}$, then the function $F:=F_{B^{n-1}}$ in $B_R$ can be decoupled into the functions $F_\theta$ corresponding to the smallest possible scale, as follows:}
\begin{equation*}
    \|F\|_{L^p(B_R)} \leq C_{\epsilon} R^{O(\epsilon)} \left(\sum_{\theta\in \mathcal{D}_R:\;d(\theta)\sim R^{-1/2}} \|F_{\theta}\|_{L^p(B_R)}^2\right)^{1/2}\;\;\;\text{ for all }2\leq p\leq \frac{2(n+1)}{n-1}.
\end{equation*}

The \textit{decoupling axioms} (DA1) and (DA2) for a sequence $(F_\tau)_{\tau\in\mathcal{D}_R}$ are the following statements.

\textbf{(DA1)} (Local constancy).  For every $\tau\in\mathcal{D}_R$ with $d(\tau)\leq R^{-\epsilon}$, the function $|F_{\tau}|$ is essentially constant on each translate of 
\begin{equation*}
    \Sigma(\tau)^{\star}:=\{x: |x \cdot (\xi-\xi_{\tau})|\leq 1 \text{ for all } \xi \in \Sigma(\tau)\},
\end{equation*}
where $\xi_\tau$ denotes the centre of $\Sigma(\tau)$.\footnote{Formally, a function is essentially constant on translates of $\Sigma(\tau)^\star$ if it satisfies estimate (24) in the statement of Lemma 6.1 in \cite{GWZ22}, with $\theta$ replaced by the smallest rectangle containing $\Sigma(\tau)$.}

\textbf{(DA2)} (Local $L^2$--orthogonality). Let $\gamma\in D_R$, and suppose that $\gamma=\sqcup_{\tau\in\mathfrak{T}}\;\tau$, where $\mathfrak{T}$ is a family of finitely overlapping caps in $\mathcal{D}_R$ with diameters smaller than $d(\gamma)$. Then, the estimate
\begin{equation*}
    \int_K |F_{\gamma}|^2 \sim \sum_{\tau \subset \gamma} \int_K |F_{\tau}|^2 + {\rm RapDec}_\epsilon(R)\int |F_\gamma|^2
\end{equation*}
holds for every convex $K\subset \mathbb{R}^n$ such that the sets $\tau+K^{\star}$, over all $\tau\in\mathfrak{T}$, are finitely overlapping.\footnote{Without (DA2), no relationship between the different $F_\tau$ would be imposed. Observe that, in contrast to the case where $(F_\tau)_{\tau\in\mathcal{D}_R}=(Eg_\tau)_{\tau\in\mathcal{D}_R}$, the equality $F_\gamma=\sum_{\tau\in\mathfrak{T}}F_\tau$ may not hold for a sequence $(F_\tau)_{\tau\in\mathcal{D}_R}$ satisfying the decoupling axioms.}

\vspace{0.3in}

It is not hard to see that, for all $g\in L^2(B^{n-1})$, the sequence $(Eg_\tau)_{\tau\in\mathcal{D}_R}$ satisfies (DA1) and (DA2). Guth's axiomatic decoupling statement above, together with a careful review of the proof \cite{GWZ22} of the refined decoupling Theorem \ref{theorem: refined Strichartz} (which directly led to our Theorem \ref{theorem: all weights}, or equivalently to \eqref{eq: restatement} below), reveal the following.

\textbf{Fact A.} \textbf{(DA1 \& DA2 $\Rightarrow$ MT with $\lessapprox R^{\frac{n-1}{n+1}}$-loss for all $Eg$)} \textit{The fact that}
\begin{equation*}
    (Eg_\tau)_{\tau\in\mathcal{D}_R}\text{ \textit{satisfies} (DA1) \textit{and} (DA2) \textit{for all} } g\in L^2(B^{n-1})
\end{equation*}
\textit{implies the inequality}
\begin{equation}\label{eq: restatement}
    \int_{B_R} |Eg|^2 w \leq C_\epsilon R^{\frac{n-1}{n+1}+\epsilon} \; \|Xw\|_\infty \;\frac{1}{R}\int_{B_R} |Eg|^2
\end{equation}
\textit{for all $g\in L^2(B^{n-1})$ and $w:\mathbb{R}^n\rightarrow [0,+\infty)$.}

To improve on the Mizohata--Takeuchi conjecture, one needs to reduce the lossy factor $R^{\frac{n-1}{n+1}}$ in \eqref{eq: restatement} (and ideally to remove it altogether). Up to $\approx 1$ factors, this is impossible if one insists on only using that all $(Eg_\tau)_{\tau\in\mathcal{D}_R}$ satisfy (DA1) and (DA2). Indeed, Guth \cite{Gu22} proved the following.

\textbf{Fact B.} \textbf{(DA1 \& DA2 $\not\Rightarrow$ MT with $\ll R^{\frac{n-1}{n+1}}$-loss for general $F$)} \textit{There exists $F:\mathbb{R}^n\rightarrow \mathbb{C}$ with}
\begin{equation}\label{eq: axioms}
    F=F_{B^{n-1}}\textit{ for some }(F_\tau)_{\tau\in\mathcal{D}_R}\textit{ satisfying}\text{ (DA1) }\textit{and}\text{ (DA2)}\textit{,}
\end{equation}
\textit{such that}
\begin{equation}\label{eq: counterexample}
    \int_{B_R}|F|^2 w \gtrsim (\log R)^{-3} R^{\frac{n-1}{n+1}} \;\|Xw\|_\infty\;\frac{1}{R}\int_{B_R}|F|^2
\end{equation}
\textit{for some $w:\mathbb{R}^n\rightarrow [0,+\infty)$.}

\begin{proof} Let $\Sigma$ be as earlier. The scale $R^{-\frac{1}{n+1}}$ plays a key role in the upcoming argument; thus, denote by $\mathcal{D}$ the set of all $\tau\in\mathcal{D}_R$ with $d(\tau)= R^{-\frac{1}{n+1}}$ (or, precisely, with $d(\tau)$ equal to the smallest dyadic number that is at least $R^{-\frac{1}{n+1}}$). For each $\tau\in\mathcal{D}$, let $\mathbb{T}_\tau$ be a family of finitely-overlapping parallel tubes in $\mathbb{R}^n$ that intersect and cover $B_R$, of radius $R^{\frac{1}{n+1}}$, length $R^{\frac{2}{n+1}}$ and direction the normal to $\Sigma(\tau)$ (these tubes are essentially translates of $\Sigma(\tau)^\star$). Let
\begin{equation*}
    \mathbb{T}:=\{T\in\mathbb{T}_\tau:\;\tau\in\mathcal{D}\}.
\end{equation*}
There exists a weight $w:\mathbb{R}^n\rightarrow [0,+\infty)$ such that the following hold.
\begin{enumerate}
    \item $w$ is the characteristic function of a union of $\sim R^{n-1}$ unit balls in $B_R$.
    \item Each tube $L$ of radius $1$ satisfies $w(L) \lesssim \log R$.
    \item Each tube $T\in\mathbb{T}$ satisfies $w(T) \lesssim \log R$, and fully contains every $1$-ball in ${\rm supp}\;w$ that it intersects.
\end{enumerate}
This is the weight that will feature in \eqref{eq: counterexample}, and its existence is guaranteed by prior work of the first author \cite[Theorem 3]{Ca09} on aspects of the Mizohata--Takeuchi conjecture. The details are omitted.

The function $F$ will be carefully defined as a sum of wave packets, so that it is large on a big proportion of ${\rm supp}\;w$; more precisely, on a large set $\mathcal{B}$ of unit balls in ${\rm supp}\;w$. The set $\mathcal{B}$ is the one appearing in the claim below. The proof is postponed to the end of the section. (Note that the claim would be trivial if each tube in $\mathbb{T}$ intersected and fully contained at most one $1$-ball in ${\rm supp}\;w$.)

\begin{claim} \label{claim} There exist\newline
\emph{(i)} a set $\mathcal{B}=\{B_1,\ldots,B_m\}$ of $\gtrsim (\log R)^{-2}R^{n-1}$ disjoint unit balls in ${\rm supp}\;w$, and\newline
\emph{(ii)} sets $\mathbb{T}_j\subset\mathbb{T}$ with $\#\mathbb{T}_j\gtrsim \#\mathcal{D}$ for every $j=1,\ldots,m$,\newline
such that the following hold.

\vspace{0.01in}

\emph{(P1)} The tubes in $\mathbb{T}_j$ contain $B_j$, for all $j=1,\ldots,m$.\newline
\emph{(P2)} For each $j=2,\ldots,m$, the tubes in $\mathbb{T}_j$ do not intersect any of the balls $B_1,\ldots,B_{j-1}$.
\end{claim}

We now construct a sequence $(F_\tau)_{\tau\in\mathcal{D}_R}$ of functions $F_\tau:\mathbb{R}^n\rightarrow\mathbb{C}$ as follows.
\begin{itemize}
    \item For each $\tau\in\mathcal{D}_R$ with $R^{-1/2}\lesssim d(\tau) < R^{-\frac{1}{n+1}}$, define $F_{\tau}:= d(\tau)^{\frac{n-1}{2}}\chi_{B_R}$.
    \item For $\tau\in\mathcal{D}_R$ with $d(\tau)= R^{-\frac{1}{n+1}}$ (or, precisely, for each $\tau\in\mathcal{D}$), define
\begin{equation*}
    F_{\tau}:=\sum_{T\in\mathbb{T}_\tau}  c_{T} e^{-2\pi i \langle\;\cdot\;,\; \xi_{\tau}\rangle} d(\tau)^{\frac{n-1}{2}} \phi_T,
\end{equation*}
where $\phi_T$ is a bump function on $T$ and $\xi_\tau$ is the centre of $\Sigma(\tau)$. The coefficients $c_T\in\mathbb{C}$ are defined below.
\item For $\gamma\in\mathcal{D}_R$ with $R^{-\frac{1}{n+1}} < d(\tau)\leq 2$, define
\begin{equation*}    F_{\gamma}:=\sum_{\tau\in\mathcal{D},\tau\subset\gamma} F_{\tau}.
\end{equation*}
\end{itemize}

Let $F:=F_{B^{n-1}}=\sum_{\tau\in\mathcal{D}}F_\tau$. The coefficients $c_T$ will all have modulus 1, and will be chosen below so that
\begin{equation}\label{eq: F large}
    |F|\gtrsim R^{\frac{n-1}{2(n+1)}}\text{ on }\bigcup_{B\in\mathcal{B}} B.
\end{equation}
\textit{Verifying \eqref{eq: axioms} and \eqref{eq: counterexample}.} For each $\tau\in\mathcal{D}$, $F_\tau$ is Fourier supported roughly in the smallest slab containing $\Sigma(\tau)$. It easily follows that $(F_\tau)_{\tau\in\mathcal{D}_R}$ satisfies the decoupling axioms (DA1) and (DA2). 

On the other hand, \eqref{eq: F large} and the small line occupancy of $w$ imply \eqref{eq: counterexample}, so $F$ and $w$ do not respect the numerology of the Mizohata--Takeuchi conjecture. Indeed,
\begin{equation*}
    \int_{B_R}|F|^2 w\gtrsim R^{\frac{n-1}{n+1}}\#\mathcal{B}\gtrsim (\log R)^{-2}R^{\frac{n-1}{n+1}}R^{n-1}
\end{equation*}
by \eqref{eq: F large}, while
\begin{eqnarray*}
\begin{aligned}
    \int |F|^2 & \lesssim \sum_{\tau\in\mathcal{D}}|F_\tau|^2\lesssim \sum_{\tau\in\mathcal{D}}\sum_{T\in\mathbb{T}_\tau}\int_T |c_T d(\tau)^{\frac{n-1}{2}}|^2\sim \sum_{\tau\in\mathcal{D}}\sum_{T\in\mathbb{T}_\tau} |T|\cdot |\tau|=|B^{n-1}|\cdot |B_R|\sim R^n
\end{aligned}
\end{eqnarray*}
due to the essential disjointness of the Fourier supports of the $F_\tau$, and therefore
\begin{equation*}
    \|Xw\|_\infty\frac{1}{R}\int_{B_R}|F|^2\lesssim (\log R) R^{n-1}\lesssim (\log R)^3 R^{-\frac{n-1}{n+1}}\int_{B_R}|F|^2 w.
\end{equation*}

\textit{Defining the $c_T$.} For $T\in\mathbb{T}$, let $\tau(T)$ be the cap $\tau\in\mathcal{D}$ with $T\in\mathbb{T}_\tau$. For $B\in\mathcal{B}$, let
\begin{equation*}
    \mathbb{T}_B:=\{T\in\mathbb{T}:\;T\text{ intersects }B\},
\end{equation*}
and observe that, once the $c_T$ are defined for all $T\in\mathbb{T}$, it will hold that
\begin{equation*}
    F_{\;|B}=R^{\frac{-(n-1)}{2(n+1)}} \sum_{T\in\mathbb{T}_B} c_T e^{-2\pi i \langle\;\cdot\;,\; \xi_{\tau(T)}\rangle} \phi_{T_{\;|B}}\;\;\;\text{ for all }B\in\mathcal{B}.
\end{equation*}
The $c_T$ are thus defined via an iteration, the $j$-th step of which ensures that the above sum has large magnitude for $B=B_j$. First, for all $T\in\mathbb{T}_{B_1}$ define
\begin{equation*}
    c_T:=e^{2\pi i \langle x_1, \; \xi_{\tau(T)}\rangle},
\end{equation*}
where $x_1$ is the centre of $B_1$. Due to the small radius of $B_1$, 
\begin{equation*}
    {\rm Re}\left(c_T e^{-2\pi i \langle x,\; \xi_{\tau(T)}\rangle}\right)={\rm Re}\left(e^{2\pi i \langle x_1-x,\; \xi_{\tau(T)}\rangle}\right)\gtrsim 1\text{ for all }x\in B_1,
\end{equation*}
hence
\begin{equation*}
    {\rm Re}\left(R^{\frac{-(n-1)}{2(n+1)}} \sum_{T\in\mathbb{T}_{B_1}} c_T e^{-2\pi i \langle\;\cdot\;,\; \xi_{\tau(T)}\rangle} \phi_T\right)\gtrsim R^{\frac{-(n-1)}{2(n+1)}} \#\mathbb{T}_1\gtrsim R^{\frac{n-1}{2(n+1)}}
\end{equation*}
on $B_1$. Therefore, once the remaining $c_T$ have been defined, we will have that
\begin{equation*}
    |F|\geq {\rm Re}\; F\gtrsim R^{\frac{n-1}{2(n+1)}}\text{ on }B_1,
\end{equation*}
as desired.

Now, fix $j=2,\ldots,m$. Suppose that, for each $i=1,\ldots,j-1$, we have performed the $i$-th step of the iteration, by defining $c_T$ for all $T\in\mathbb{T}_{B_1}$ (when $i=1$) and for all $T\in\mathbb{T}_{B_i}\setminus (\mathbb{T}_{B_1}\cup\ldots\mathbb{T}_{B_{i-1}})$ (when $i\geq 2$) so that
\begin{equation*}
    \left|{\rm Re}\left(R^{\frac{-(n-1)}{2(n+1)}} \sum_{T\in\mathbb{T}_{B_i}} c_T e^{-2\pi i \langle\;\cdot\;,\; \xi_{\tau(T)}\rangle} \phi_T\right)\right|\gtrsim R^{\frac{n-1}{2(n+1)}}
\end{equation*}
on $B_i$ (which ensures that, once the remaining $c_T$ have been defined, we will have that
\begin{equation*}
    |F|\gtrsim R^{\frac{n-1}{2(n+1)}}\text{ on }B_1,\ldots,B_{j-1}).
\end{equation*}
During the $j$-th step of the iteration, we will define the $c_T$ for $T\in\mathbb{T}_{B_j}\setminus (\mathbb{T}_{B_1}\cup\ldots\cup\mathbb{T}_{B_{j-1}})$ so that
\begin{equation*}
    \left|{\rm Re}\left(R^{\frac{-(n-1)}{2(n+1)}} \sum_{T\in\mathbb{T}_{B_j}} c_T e^{-2\pi i \langle\;\cdot\;,\; \xi_{\tau(T)}\rangle} \phi_T\right)\right|\gtrsim R^{\frac{n-1}{2(n+1)}}
\end{equation*}
on $B_j$ (ensuring that eventually $|F|\gtrsim R^{\frac{n-1}{2(n+1)}}$ on $B_j$ as well). Write
\begin{equation*}
    \mathbb{T}_{B_j}:=\mathbb{T}_{B_j}^1\sqcup \mathbb{T}_{B_j}^2,
\end{equation*}
where $\mathbb{T}_{B_j}^1:=\mathbb{T}_{B_j}\setminus(\mathbb{T}_{B_1}\cup\ldots\mathbb{T}_{B_{j-1}})$ (the set of tubes through $B_j$ for which we still need to define the $c_T$), while $\mathbb{T}_{B_j}^2$ consists of the tubes through $B_j$ for which the $c_T$ have already been defined. Importantly, $\mathbb{T}_{B_j}^1\supset\mathbb{T}_j$.

Let $\sigma_{B_j}$ be the sign of $F^2_j:={\rm Re}\left(\sum_{T\in\mathbb{T}_{B_j}^2}c_T e^{-2\pi i \langle\;\cdot\;,\; \xi_{\tau(T)}\rangle} \phi_T\right)$ on $B_j$ \footnote{Technically, this sign does not have to be uniform over all points of $B_j$; we can however choose the dominant sign over $B_j$, and eventually control the sum of the $F_\tau$ on a large subset of $B_j$. We omit this additional technicality from our exposition.}, and define
\begin{equation*}
    c_T:=\sigma_{B_j}e^{2\pi i \langle x_j,\xi_{\tau(T)}\rangle}\text{ for all }T\in\mathbb{T}_{B_j}^1,
\end{equation*}
where $x_j$ is the centre of $B_j$. As earlier,
\begin{equation*}
    \left|{\rm Re}\left(c_T e^{-2\pi i \langle \;\cdot\;,\; \xi_{\tau(T)}\rangle}\right)\right|\gtrsim 1\text{ on }B_j;
\end{equation*}
and, crucially, ${\rm Re}\left(c_T e^{-2\pi i \langle \;\cdot\;,\; \xi_{\tau(T)}\rangle}\right)$ also has sign $\sigma_{B_j}$ on $B_j$, for all $T\in\mathbb{T}^1_{B_j}$. Therefore, the functions $F_j^2$ and
\begin{equation*}
    F_j^1:={\rm Re}\left(R^{\frac{-(n-1)}{2(n+1)}} \sum_{T\in\mathbb{T}^1_{B_j}} c_T e^{-2\pi i \langle\;\cdot\;,\; \xi_{\tau(T)}\rangle} \phi_T\right)
\end{equation*}
have the same sign on $B_j$, so
\begin{eqnarray*}
\begin{aligned}
    \left|{\rm Re}\left(R^{\frac{-(n-1)}{2(n+1)}} \sum_{T\in\mathbb{T}_{B_j}} c_T e^{-2\pi i \langle\;\cdot\;,\; \xi_{\tau(T)}\rangle} \phi_T\right)\right|&=|F_j^1+F_j^2|\geq |F_j^1|\\
    &\gtrsim R^{\frac{-(n-1)}{2(n+1)}} \#\mathbb{T}_j\gtrsim R^{\frac{n-1}{2(n+1)}}
\end{aligned}
\end{eqnarray*}
on $B_j$, as desired. 

For all $T\in\mathbb{T}$ that do not contain any of the balls in $\mathcal{B}$, we define $c_T=1$.
By the end of the iteration, \eqref{eq: F large} holds. 
\end{proof}

\begin{proof}[Proof of Claim \ref{claim}.] Let $\mathcal{P}$ be a family of disjoint unit balls inside ${\rm supp}\;w$, with
\begin{equation*}
    \#\mathcal{P}\sim |{\rm supp}\;w|\sim R^{n-1}.
\end{equation*}
For each $B\in \mathcal{P}$, denote by $\mathbb{T}_B$ the set of tubes in $\mathbb{T}$ through $B$, and observe that $\#\mathbb{T}_B=\#\mathcal{D}$.

Write $\mathcal{P}=\{B_1,B_2,\ldots,B_N\}$. To prove the claim, we will show that there exist indices $k_1<k_2<\ldots<k_m$ such that:
\begin{itemize}
\item $m\gtrsim (\log R)^{-10} R^{n-1}$,
\item $B_{k_1}=B_1$, and for each $j=2,3,\ldots,m$, at least $\#\mathcal{D}/2$ tubes in $\mathbb{T}_{B_{k_j}}$ do not lie in $\mathbb{T}_{B_{k_1}}\cup\mathbb{T}_{B_{k_2}}\cup\ldots\cup\mathbb{T}_{B_{j-1}}$.
\end{itemize}
Indeed,
\begin{itemize}
    \item let $k_1:=1$,
    \item let $k_2$ be the smallest $j>k_1$ such that at most $\#\mathcal{D}/2$ tubes through $B_j$ contain $B_{k_1}$,
    \item let $k_3$ be the smallest $j>k_2$ such that at most $\#\mathcal{D}/2$ tubes through $B_j$ contain $B_{k_1}$ or $B_{k_2}$,
\end{itemize}
and so on, until no further $k_j$ as above exists. Let $\mathcal{P}^1$ be the set of balls $B_{k_j}$, over all the $k_j$ selected via the above process. To complete the proof of the claim, it will now be shown that
\begin{equation*}
    \#\mathcal{P}^1\gtrsim (\log R)^{-2} R^{n-1},
\end{equation*}
by studying the incidences between $\mathcal{P}$ and $\mathbb{T}$. For any $\mathcal{S}\subset \mathcal{P}$ and $\mathbb{L}\subset \mathbb{T}$, denote
\begin{equation*}
    I(\mathcal{S},\mathbb{L}):=\#\{(B,T)\in S\times\mathbb{L}:\;B\text{ is contained in }T\},
\end{equation*}
the number of incidences between $\mathcal{S}$ and $\mathbb{L}$.

Assume for contradiction that 
\begin{equation}\label{eq: implicit}
    \#\mathcal{P}^1\lesssim (\log R)^{-2} \#\mathcal{P}
\end{equation}
for an appropriately small implicit constant. Then, the set $\mathbb{T}^1$ of tubes in $\mathbb{T}$ that pass through balls in $\mathcal{P}^1$ is not too large; in particular,
\begin{equation*}
    \#\mathbb{T}^1\leq I(\mathcal{P}^1,\mathbb{T}^1)\leq \mathcal{P}^1\#\mathcal{D}\lesssim (\log R)^{-2} \#\mathcal{P}\#\mathcal{D}\sim (\log R)^{-2} I(\mathcal{P},\mathbb{T})\lesssim (\log R)^{-1}\#\mathbb{T},
\end{equation*}
for a small implicit constant. Therefore, the tubes in $\mathbb{T}^1$ only contribute a small fraction of the total incidences between $\mathbb{T}$ and $\mathcal{P}$:
\begin{equation*}
    I(\mathcal{P},\mathbb{T}^1)\lesssim \#\mathbb{T}^1\log R\lesssim \mathbb{T}\sim (\log R)^{-1}\#\mathcal{D}R^{n-1}\sim \#\mathcal{D}\#\mathcal{P}\leq \frac{1}{10}\;I(\mathcal{P},\mathbb{T})
\end{equation*}
(the implicit constant in \eqref{eq: implicit} is chosen so that this is true). 

This is a contradiction, as $\mathcal{P}^1$ was selected so that $\mathbb{T}^1\;(=\cup_{j=1}^m\mathbb{T}_{B_{k_j}})$ contributes at least half of the total incidences between $\mathbb{T}$ and $\mathcal{P}$. Indeed, each $B_i\in \mathcal{P}\setminus \mathcal{P}^1$ is incident to at least $\#\mathcal{D}/2$ tubes in $\cup_{k_j<i}\mathbb{T}_{B_{k_j}}\subset \mathbb{T}^1$; while each $B_i\in\mathcal{P}$ has all the $\#\mathcal{D}$ tubes in $\mathbb{T}$ through it in $\mathbb{T}^1$. Therefore,
\begin{equation*}
    I(\mathcal{P},\mathbb{T}^1)\geq  \#\mathcal{P}\#\mathcal{D}/2=\frac{1}{2}\;I(\mathcal{P},\mathbb{T}),
\end{equation*}
contradicting \eqref{eq: implicit}.
\end{proof}


\begin{thebibliography}{999999999}

\bibitem[BBC08]{BBC08}
J. A. Barcel\'o, J. Bennett, A. Carbery, `A note on localised weighted inequalities for the extension operator', \textit{J. Austr. Math. Soc.} \textbf{84} (2008), 289-299.

\bibitem[BRV97]{BRV97}
J. A. Barcel\'o, A. Ruiz, L. Vega, `Weighted estimates for the Helmholtz equation and consequences',
\textit{Journal of Functional Analysis} \textbf{150}, no. 2 (1997), 356-382.

\bibitem[BCSV06]{BCSV06}
J. Bennett, A. Carbery, F. Soria, A. Vargas, `A Stein conjecture for the circle', \textit{Math. A{n}nal.} \textbf{336} (2006), 671-695.

\bibitem[BN21]{BN21}
J. Bennett, S. Nakamura, `Tomography bounds for the Fourier extension operator and applications', \textit{Math. Annal.} \textbf{30} (2021), 119-159.

\bibitem[BNS22]{BNS22}
J. Bennett, S. Nakamura, S. Shiraki, `Tomographic Fourier extension identities for submanifolds in $\mathbb{R}^n$' (2022), \textit{arXiv:2212.12348}.

\bibitem[B93]{B93}
J. Bourgain, `Fourier transform restriction phenomena for certain lattice subsets and applications to
nonlinear evolution equations. I. Schr\"{o}dinger equations', \textit{Geom. Funct. Anal.} \textbf{3} (1993), 107-156

\bibitem[B94]{B94}
J. Bourgain, `Hausdorff dimension and distance sets' (1994), \textit{Israel J. Math.} \textbf{87} (1994), 193-201.

\bibitem[BD15]{BD15}
J. Bourgain, C. Demeter, `The proof of the $\ell^2$ decoupling conjecture', \textit{Ann. of Math. (2)} \textbf{182} (2015), 351-389.
  
\bibitem[Ca09]{Ca09}
A. Carbery, `Large sets with limited tube occupancy', \textit{J. London Math. Soc.}  (2) 79 (2009), 529–543.

\bibitem[CRS92]{CRS92}
A. Carbery, E. Romera, F. Soria, `Radial weights and mixed norm
inequalities for the disc multiplier', \textit{J. Funct. Anal.} \textbf{109} (1992), 52-75.

\bibitem[CS00]{CS00}
A. Carbery, A. Seeger, `Weighted inequalities for {B}ochner-{R}iesz means in the plane', \textit{Q. J. Math.} \textbf{51} (2000), 155-167.

\bibitem[CS97a]{CS97a}
A. Carbery, F. Soria,
`Sets of divergence for the localisation problem for {F}ourier integrals', \textit{C. R. Acad. Sci. Paris S\'{e}r. I Math.}, \textbf{325}, (1997), 1283-1286.
     
\bibitem[CS97b]{CS97b}
A. Carbery, F. Soria, `Pointwise {F}ourier inversion and localisation in $\mathbb{R}^n$', \textit{J. Fourier Anal. Appl.}, \textbf{3}, (1997), 847-858.

\bibitem[CSV07]{CSV07}
A. Carbery, F. Soria, A. Vargas, `Localisation and weighted inequalities for spherical {F}ourier
means', \textit{J. Anal. Math.}, \textbf{103}, (2007), 133-156.

\bibitem[DGO$^+$21]{DGOWWZ21}
X. Du, L. Guth, Y. Ou, H. Wang, B. Wilson, R. Zhang, `Weighted restriction estimates and application to {F}alconer distance set problem', \textit{Amer. J. Math.},
\textbf{143}, (2021), 175-211.

\bibitem[DZ19]{DZ19}
X. Du, R. Zhang, `Sharp $L^2$ estimate of the {S}chr\"odinger maximal function in higher dimensions', 
\textit{Ann. Math.}, \textbf{189}, (2019), 837-861.

\bibitem[E04]{E04}
B. Erdo\~{g}an, `A note on the {F}ourier transform of fractal measures', 
\textit{Math. Res. Lett.}, \textbf{11}, (2004), 299-313.

\bibitem[GWZ22]{GWZ22}
S. Guo, H. Wang, R. Zhang, `A dichotomy for H\"ormander-type
oscillatory integral operators', \textit{arXiv:2210.05851}.

\bibitem[Gu18]{Gu18}
L. Guth, `Restriction estimates using polynomial partitioning II', \textit{Acta Math.} \textbf{221}, no. 1 (2018), 81-142.

\bibitem[Gu22]{Gu22}
L. Guth, `An enemy scenario in restriction theory', joint talk for AIM Research Community \textit{Fourier restriction conjecture and related problems} and HAPPY network (2022), \url{https://www.youtube.com/watch?v=x-DET83UjFg}.

\bibitem[GIOW20]{GIOW20}
L. Guth, A. Iosevich, Y. Ou, H. Wang, `On Falconer’s distance set problem in the plane', \textit{Invent. Math.} \textbf{219} (2020), no. 3, 779-830.

\bibitem[GMW20]{GMW20}
L. Guth, D. Maldague, H. Wang, `Improved decoupling for the parabola', \textit{arXiv:2009.07953}.

\bibitem[GWZ20]{GWZ20}
L. Guth, H. Wang, R. Zhang, `A sharp square function estimate for the cone in $\mathbb{R}^3$', \textit{Ann. of Math. (2)} \textbf{192} (2020), 551-581.

\bibitem[HI22]{HI22}
J. Hickman, M. Iliopoulou, `Sharp $L^p$ estimates for oscillatory integral operators of arbitrary signature', \textit{Math. Z.}, \textbf{301}, no. 1 (2022), 1143-1189.

\bibitem[Mi85]{Mi85}
S. Mizohata, `On the Cauchy Problem', \textit{Notes and Reports in Mathematics, Science and Engineering}
\textbf{3} (1985), Academic Press, San Diego, CA.

\bibitem[R86]{R86}
N. N. Rogovskaya,
`An asymptotic formula for the number of solutions of a
system of equations. Diophantine approximation, Part II' (Russian), \textit{Moskov.
Gos. Univ. Moscow},  \textbf{64} (1986), 78-84.

\bibitem[Sh21]{Sh21}
B. Shayya,
`Fourier restriction in low fractal dimensions', \textit{Proc. Edinb. Math. Soc. (2)},  \textbf{64} (2021), 373-407.
     
\bibitem[Sh22]{Sh22}
B. Shayya, `Mizohata--Takeuchi estimates in the plane', \textit{Bull. Lond. Math. Soc. (5)}, \textbf{55} (2023).

\bibitem[St78]{St78}
E. M. Stein, `Some problems in harmonic analysis', \textit{Proc. Sympos. Pure Math.}, Williamstown, Mass. \textbf{1}
(1978), 3-20.

\bibitem[V81]{V81}
R. C. Vaughan, `The {H}ardy-{L}ittlewood method', \textit{Cambridge Tracts in Mathematics} \textbf{80}, Cambridge University Press (1981).

\bibitem[V97]{V97}
R. C. Vaughan, `The {H}ardy-{L}ittlewood method', 2nd edition, \textit{Cambridge Tracts in Mathematics}, \textbf{125}, Cambridge University Press (1981).
\end{thebibliography}
\end{document}